\documentclass[12pt,a4paper]{article}

\usepackage[utf8]{inputenc}
\usepackage[english]{babel}
\usepackage{amsmath,amssymb,amsfonts,amsthm,graphicx}
\usepackage{hyperref}

\begin{document}

\newcommand{\C}{\mathbb {C}}
\newcommand{\N}{\mathbb{N}}
\newcommand{\R}{\mathbb{R}}
\newcommand{\tf}{\mathcal{F}}

\renewcommand{\theenumi}{\roman{enumi}}
\renewcommand{\labelenumi}{\theenumi)}

\swapnumbers
\newtheorem{EXPLE}{Examples}[section]
\newtheorem{VDC}[EXPLE]{Theorem}
\newtheorem{REM1}[EXPLE]{Examples}
\newtheorem{Details}[EXPLE]{Proofs of technical details}
\newtheorem{REM2}[EXPLE]{Remark}
\newtheorem{VDC2}[EXPLE]{Theorem}
\newtheorem{REM3}[EXPLE]{Remark}
\newtheorem{VDC3}[EXPLE]{Theorem}
\newtheorem{VDC3bis}[EXPLE]{Remark}
\newtheorem{COR1}[EXPLE]{Corollary}
\newtheorem{REM2bis}[EXPLE]{Remark}
\newtheorem{VDC4}[EXPLE]{Theorem}
\newtheorem{VDC4bis}[EXPLE]{Remarks}
\newtheorem{DEF1}{Definition}[section]
\newtheorem{REM8}[DEF1]{Remark}
\newtheorem{SCHRO1}[DEF1]{Theorem}
\newtheorem{SCHRO2}[DEF1]{Corollary}
\newtheorem{OPTI}[DEF1]{Theorem}
\newtheorem{REM9}[DEF1]{Remark}
\newtheorem{PREGEO1}[DEF1]{Theorem}
\newtheorem{PREGEO2}[DEF1]{Theorem}
\newtheorem{REM4}[DEF1]{Remarks and Example}
\newtheorem{SCHRO3}[DEF1]{Theorem}
\newtheorem{REM5}[DEF1]{Remark}
\newtheorem{REM7}[DEF1]{Remark}
\newtheorem{SCHRO4}[DEF1]{Theorem}
\newtheorem{REM6}[DEF1]{Remarks}
\newtheorem{LEM}{Lemma}[section]
\newtheorem{LEM2}[LEM]{Lemma}

\title{Asymptotic estimates of oscillatory integrals with general phase and singular amplitude:\\ Applications to dispersive equations}

\author{Florent Dewez\footnote{Université Lille 1, Laboratoire Paul Painlevé, CNRS U.M.R 8524, 59655 Villeneuve d'Ascq Cedex, France. Email: florent.dewez@math.univ-lille1.fr}}

\date{}

\maketitle

\begin{abstract}

	In this paper, we furnish van der Corput types estimates for oscillatory integrals with respect to a large parameter, where the phase is allowed to have a stationary point of real order and the amplitude to have an integrable singularity. The resulting estimates show explicitly the influence of these two particular points on the decay. These results are then applied to the solutions of a family of dispersive equations whose generators are Fourier multipliers. We explore the effect of a limitation to compact frequency bands and of singular frequencies of the initial condition on the decay. Uniform estimates in space-time cones as well as $L^{\infty}$-norm estimates are furnished and the optimality of the decay rates is proved under certain hypotheses. Moreover the influence of a growth limitation at infinity of the symbols on the dispersion is exhibited.

\end{abstract}

\vspace{0.3cm}

\noindent \textbf{Mathematics Subject Classification (2010).} Primary 35B40; Secondary 35S10, 35B30, 35Q41, 35Q40.

\noindent \textbf{Keywords.} van der Corput, singular frequency, Fourier multiplier, dispersive equations, space-time cone, (optimal) $L^{\infty}$-time decay.

\setcounter{section}{-1}
\section{Introduction}

\hspace{2.5ex} The asymptotic behaviour of solutions of dispersive equations can sometimes be derived from estimates of oscillatory integrals with respect to a large parameter. The van der Corput Lemma \cite[p. 332]{stein} permits to control oscillatory integrals in terms of the phase function and the amplitude if they are regular, exhibiting the decay when the large parameter tends to infinity. In view of applications to certain dispersive equations, many adaptations of this result were developed in the literature (see for example \cite{liess}, \cite{benartzi1994}, \cite{ionescu}). In the present paper, an extension is proposed: the phase function is allowed to have a stationary point of real order and a frequency of the compact integration interval can be an integrable singularity of the amplitude. We provide estimates showing in an explicit way the influence of the order of the stationary point and of the strength of the singularity on the decay. The results are then applied to the solutions of a family of evolution equations, whose generators are Fourier multipliers. For specific initial conditions, we furnish $L^{\infty}$-norm estimates as well as uniform estimates in certain space-time cones. These results show in which way the initial data and the symbol may affect the dispersion of the solution. In particular, this paper combined with \cite{article1} and \cite{article2} solves the open problem of the optimal $L^{\infty}$-decay rate for the free Schrödinger equation on the line with initial conditions whose Fourier transforms have certain absolutely integrable singularities.

Consider the following evolution equation
\begin{equation*}
	\left\{ \begin{array}{rl}
			& \hspace{-2mm} \left[ i \, \partial_t - f \big(D) \right] u(t,x) = 0 \\ [2mm]
			& \hspace{-2mm} u(0,x) = u_0(x)
	\end{array} \right. \; ,
\end{equation*}
for $t > 0$ and $x \in \R$, where the symbol $f$ belongs to $\mathcal{C}^{\infty}(\R)$ and all the derivatives grow at most as a polynomial at infinity. In \cite{benartzi1994}, the authors consider symbols of the form $f(p) = |p|^{\rho} + R(p)$, where $\rho \geqslant 2$ and $R$ is a regular real-valued function whose growth at infinity is controlled in a certain sense by $|p|^{\rho-1}$. They remark that the operator $u_0 \in L^2(\R) \longmapsto u(t,.) \in L^2(\R)$ is unitary for all $t > 0$ and, using a van der Corput type lemma, they establish the following estimate,
\begin{equation*}
	\big\| u(t,.) \big\|_{L^{\infty}(\R)} \leqslant C_1 \, t^{-\frac{1}{\rho}} \| u_0 \|_{L^1(\R)} \; ,
\end{equation*}
for a certain constant $C_1 > 0$, showing the dispersive nature of the equation. A Strichartz type estimate is then derived, leading to the following $L^{\infty}$-norm estimate in the case $u_0 \in L^2(\R)$,
\begin{equation*}
	\big\| u(t,.) \big\|_{L^{\infty}(\R)} \leqslant C_2 \, t^{-\frac{1}{2 \rho}} \| u_0 \|_{L^2(\R)} \; ,
\end{equation*}
for a certain constant $C_2 > 0$.\\
Now consider for example initial conditions $u_0$ satisfying
\begin{equation} \label{intro}
	\forall \, p \in \R \qquad	\tf u_0 (p) = p^{\mu-1} \, \chi_{[0,1]}(p) \; ,
\end{equation}
with $\mu \in (0,1]$; here $\tf u_0$ refers to the Fourier transform of $u_0$ and $\chi_{[0,1]}$ is the characteristic function of $[0,1]$. Under this assumption, $u_0$ is a smooth function which never belongs to $L^1(\R)$ and belongs to $L^2(\R)$ if and only if $\mu \in \big(\frac{1}{2}, 1\big]$. The above results do not treat such an initial data when $\mu \in \big(0,\frac{1}{2}\big]$ and hence the question of the $L^{\infty}$-time decay rate for the above problem when $\mu \in \big(0,\frac{1}{2}\big]$ seems to be open.\\
Now let us recall the results that we obtained in \cite{article1} and \cite{article2}. We considered the above evolution equation in the case $f(p) = p^2$, in other words the free Schrödinger equation on the line. We explored the influence of initial conditions in a compact frequency band with a singular frequency at one of the endpoints of the band (as for example \eqref{intro}) on the dispersion. Employing a slightly improved version of the stationary phase method of A. Erdélyi \cite[Section 2.8]{erdelyi}, we provided asymptotic estimates to one term of the solution and showed that it tends to be concentrated in a space-time cone generated by the frequency band; this phenomenon was already pointed out in \cite{fam2012} in another context. Furthermore, we found out that the singular frequency has globally an effect on the decay, in particular in regions containing the direction given by the singularity. The stationary phase method permits to obtain the optimal decay rates in certain space-time cones but can not cover the whole space-time, since certain regions remain uncontrolled. This is due to the fact that the first terms and the remainder of the asymptotic estimates to one term are not uniformly bounded.\\
Our aim in this paper is to complete our preceding results by estimating the solution in arbitrary space-time cones, even in cones containing the forbidden regions of \cite{article1} and \cite{article2}, as well as in the whole space-time, in the case of general symbols $f$ verifying $f'' \neq 0$. In particular, the general results furnish an answer to the above open problem for initial conditions $u_0$ in a compact frequency band with an integrable singular frequency, including the case of the free Schrödinger equation which was partially solved in \cite{article1} and \cite{article2}.

In Section 1, we consider oscillatory integrals with respect to the large parameter $\omega$ of the type
\begin{equation} \label{oscillatory}
	\int_{p_1}^{p_2} U(p) \, e^{i \omega \psi(p)} \, dp \; .
\end{equation}
The amplitude $U$ can be singular at $p_1$; factorizing the singularity, $U$ can be written as follows,
\begin{equation} \label{amplitude}
	U(p) = (p-p_1)^{\mu-1} \, \tilde{u}(p) \; ,
\end{equation}
where $\mu$ is supposed to belong to $(0,1]$ and $\tilde{u}$ is called the regular part of the amplitude.  The phase function $\psi$ is  allowed to have a unique stationary point $p_0$ of real order; more precisely, we suppose the factorization
\begin{equation} \label{phase}
	\psi'(p) = | p - p_0 |^{\rho-1} \, \tilde{\psi}(p) \; ,
\end{equation}
where $\rho \in \R$ is larger than $1$ and $\tilde{\psi}$, which satisfies $\big| \tilde{\psi} (x) \big| > 0$ for any $x \in [p_1,p_2]$, is called the non-degenerate part of the phase. For example, smooth functions with vanishing first derivatives are included. The idea of supposing these factorizations, which is well suited for the formulation of the results, has been inspired by \cite{erdelyi}. \\
We furnish van der Corput type estimates of these integrals with explicit dependence of the decay rate on the order of the stationary point and of the singularity. In the first results, the phase function is assumed to have a stationary point $p_0$ which is either inside or outside the domain of integration, and we give estimates of the above integral which are uniform with respect to the position of $p_0$. To establish these uniform estimates, we combine the classical method \cite[p. 332]{stein} with the above hypotheses of the phase and of the amplitude. This procedure takes into account the possibility that the derivative of the phase function can be vanishing inside the domain of integration but also arbitrarily close to zero if the stationary point is outside and potentially close to $[p_1,p_2]$. In the last result, we furnish another estimate of the oscillatory integral in the case of the absence of a stationary point inside $[p_1,p_2]$; here we give up the uniformity of the constant in view of a better decay rate. In this situation, the assumptions of the phase can be relaxed as compared with the two previous results.

In Section 2, we consider Fourier solution formulas for a class of initial value problems defined by Fourier multipliers. We are interested in the influence of the initial data and of the symbol on the decay. Thanks to a rewriting of the solution as an oscillatory integral with respect to time, inspired by \cite{article1} and \cite{article2}, the results of Section 1 are applicable. Firstly we show that an initial data having a Fourier transform with a compact support diminishes the time-decay of the solution in a cone related to this support. Then the influence of a singular frequency is explored: the decay rate is strongly affected in cones containing the space-time direction given by the singular frequency, diminishing the $L^{\infty}$-norm time decay. Finally we show that a limited growth of the symbol of the Fourier multiplier restricts the domain of influence of the stationary point on the decay. It follows that the solution tends faster to $0$ outside a certain cone, which depends only on the symbol, than inside.

In the last section, we state and prove two basic lemmas which are used in Section 1 and Section 2.

Finally let us comment on the related literature. In \cite{reed-simon}, the authors state the $L^1-L^{\infty}$ estimate for the unitary group generated by the free Hamiltonian. Using interpolation, they obtain the well-known $L^p-L^q$ estimates for the free Schrödinger equation on $\R^n$.

The first Strichartz type estimates are established in \cite{strichartz} in the case of the free Schrödinger equation and the wave equation. Using complex analysis, the author provides estimates of the $L^2(S)$-norm of Fourier transforms of functions belonging to $L^q(\R^n)$, for some $q \geqslant 1$, where $S$ is a quadratic surface. These considerations lead to the above mentioned estimates.

The authors of \cite{msw} apply the van der Corput lemma to the solution formulas of the wave equation equation and the Klein-Gordon equation on $\R^n$ to derive $L^{\infty}$-estimates.

A similar result was obtained in \cite{fam94} in the case of the Klein-Gordon equation on $\R$ with constant but different potential on the two half-axes. The author uses a spectral theoretic formula in order to apply the van der Corput lemma to the solution of the equation.

A generalization of both mentioned results \cite{reed-simon} and \cite{strichartz} is provided in \cite{benartzi1994} where the generator of the evolution equation is a Fourier multiplier with a symbol belonging to a certain class. Van der Corput type estimates are established in the case of oscillatory integrals defined in a weak sense and are then employed to study the dispersive nature of the equation.

In \cite{liess}, the authors are interested in the decay of Fourier transforms on singular surfaces. To do so, they establish a variant of the van der Corput lemma based on Stein's result \cite{stein}. Here the $N$-th derivative of the phase is allowed to vanish at the boundary of the finite integration interval at a certain order; according to the classical van der Corput lemma, a slower decay rate than $t^{-\frac{1}{N}}$ is expected. But the authors suppose that the amplitude tends to $0$ at the stationary point with a decay related to the order of the zero of the phase. Thanks to this coupling, the amplitude limits the effect of the stationary point on the decay rate, which remains at $t^{-\frac{1}{N}}$.

The time-decay rate of the free Schrö\-din\-ger equation is considered in \cite{cazenave1998} and \cite{cazenave2010}. In \cite{cazenave1998}, singular initial conditions are constructed to derive the exact $L^p$-time decay rates of the solution, which are slower than the classical results for regular initial conditions. In \cite{cazenave2010}, the authors construct initial conditions in Sobolev spaces (based on the Gaussian function), and they show that the related solutions has no definite $L^p$-time decay rates, nor coefficients, even though upper estimates for the decay rates are established. Both papers are based on special formulas for functions and their Fourier transforms.\\
Though we do not furnish $L^p$-estimates for $p \neq \infty$, our method permits to cover a larger class of initial data and to consider more general symbols including $f(p) = p^2$.

In the setting of \cite{fam94}, the article \cite{fam2012} furnishes an asymptotic estimate to one term of the solution. The initial data are chosen in frequency bands away from the critical frequency associated with the potential step. This article has inspired the study in \cite{article1} and \cite{article2}, and presents the same uniformity problem. The present paper solves it in the setting of \cite{article1} and \cite{article2}, so we hope that our theoretical results are applicable to the setting of \cite{fam2012}.

In \cite{ionescu}, the fractional Schrödinger equation, which was firstly introduced in \cite{laskin}, is considered in one dimension. The authors furnish an $L^{\infty}$-estimate of the free solution for initial data belonging to an appropriate functional space. To do so, they estimate the solution formula given by an oscillatory integral by employing a van der Corput type method. This result combined with other technical arguments permits to study a nonlinear variant of the fractional Schrödinger equation on the line.

One can finally mention the results of \cite{fam0}. Here the authors consider the Schrödinger equation with sufficiently localized potential on a star-shaped network and provide $L^{\infty}$-decay estimates. A perturbation estimate shows that the solution is close to the free solution in the high frequencies. In particular this result is applicable to the Schrödinger equation with potential on the line and permits to transfer some quantitative information from the free equation obtained in \cite{article1} and \cite{article2} to the perturbed equation.\\

\noindent \textbf{Acknowledgements:}\\
The author thanks E. Creusé for valuable support and F. Ali Mehmeti for the helpful and numerous discussions.

\section{Stationary points of real order and singular amplitudes: van der Corput type estimates}

\hspace{0.5cm} We start by stating the hypotheses on the phase function that we shall use throughout this section. Two examples are then given to illustrate theses assumptions.\\

\noindent Let $p_1,p_2$ be two finite real numbers such that $p_1 < p_2$, and let $I$ be an open interval containing $[p_1,p_2]$.\\

\noindent \textbf{Assumption ($\mathbf{P_{p_0,\rho}}$).} Let $p_0 \in I$ and $\rho > 1$.\\ A function $\psi : I \longrightarrow \R$ satifies Assumption (P$_{p_0,\rho}$) if and only if $\psi \in \mathcal{C}^2 \big( I \big)$ and there exists a function $\tilde{\psi} : I \longrightarrow \R$ such that
	\begin{equation*}
		\forall \, p \in I \qquad \psi'(p) = |p-p_0|^{\rho -1} \, \tilde{\psi}(p) \; ,
	\end{equation*}
	where $| \tilde{\psi} |: I \longrightarrow \R$ is assumed continuous and does not vanish on $I$.\\ The point $p_0$ is called \emph{stationary point} of $\psi$ of order $\rho -1$, and $\tilde{\psi}$ the \emph{non-degenerate part} of $\psi$.\\
	
\noindent Let us comment on this choice. Firstly we want to include stationary points of non-integer order in the study, so we have to consider the absolute value of $(p-p_0)$. Secondly the assumptions that we made on the function $\tilde{\psi}$ are such that it does not contribute to the order of the stationary point $p_0$; in particular, $\tilde{\psi}$ has to be non-vanishing. But this is not sufficient in our context; indeed in the results of this section, we shall ensure that $\displaystyle \min_{[p_1,p_2]} | \tilde{\psi} |$ exists and is non-zero. The continuity of $| \tilde{\psi} |$ permits to obtain such a result. Nevertheless we don't claim that we achieve maximum generality with these hypotheses. \\ Regarding the regularity of $\tilde{\psi}$, it is interesting to note that $\tilde{\psi}$ is actually continuously differentiable on $\{ p \in I \, | \, p < p_0 \}$ and on $\{ p \in I \, | \, p > p_0 \}$, because
\begin{equation*}
	\forall \, p \neq p_0 \qquad \tilde{\psi}(p) = \frac{\psi'(p)}{|p-p_0|^{\rho-1}} \; .
\end{equation*}
This implies that $\tilde{\psi}$ has a constant sign on $\{ p \in I \, | \, p < p_0 \}$ and $\{ p \in I \, | \, p > p_0 \}$; note that the sign can be different on each interval.\\
The above Assumption (P$_{p_0,\rho}$) permits to study both following settings. In particular, the first example shows that smooth functions with vanishing first derivatives are included.

\begin{EXPLE} \label{EXPLE}
	\em \begin{enumerate}
		\item Let $\psi : I \longrightarrow \R$ belonging to $\mathcal{C}^{N} \big( I \big)$ for a certain $N \geqslant 2$, and let $p_0 \in I$. Suppose that $\psi^{(k)}(p_0) = 0$ for $k=1, \ldots, N-1$. Then by Taylor's formula, we obtain
	\begin{align*}
		\psi'(p)	& = \frac{1}{(N-2)!} \int_{p_0}^p (p-x)^{N-2} \, \psi^{(N)}(x) \, dx \\
					& = \frac{(p-p_0)^{N-1}}{(N-2)!} \int_0^1 (1-y)^{N-2} \, \psi^{(N)} \big( y(p-p_0) + p_0 \big) \, dy \; ,
	\end{align*}
	for all $p \in I$. If we define $\tilde{\psi}$ as follows
	\begin{equation*}
		\tilde{\psi}(p) := \left\{ \begin{array}{rl}
			& \hspace{-4mm} \displaystyle \frac{1}{(N-2)!} \left(\frac{p-p_0}{|p-p_0|} \right)^{N-1} \int_0^1 (1-y)^{N-2} \, \psi^{(N)} \big( y(p-p_0) + p_0 \big) \, dy \; , \quad \text{if} \; p \neq p_0 \; , \\
			& \vspace{-0.3cm} \\
			& \hspace{-4mm} \displaystyle \frac{1}{(N-1)!} \, \psi^{(N)}(p_0) \; , \quad \text{if} \; p = p_0 \; ,
		\end{array} \right.
	\end{equation*}
	then $\displaystyle \psi'(p) = |p-p_0|^{N -1} \, \tilde{\psi}(p)$. Supposing $\big| \psi^{(N)} \big| > 0$ on $I$ implies that $\psi$ satisfies Assumption (P$_{p_0,N}$).
	\item Let $N \in \N$ such that $N \geqslant 2$ and choose $\alpha \in (N-1,N)$. Suppose that $\displaystyle \psi'(p) = |p|^{\alpha}$, for all $p \in \R$. In this case, $\psi \in \mathcal{C}^N \big( \R \big)$ but $\psi \notin \mathcal{C}^{N+1} \big( \R \big)$, and $\tilde{\psi} = 1$. Then Assumption (P$_{0,\alpha}$) is satisfied.
	\end{enumerate}
\end{EXPLE}

\vspace{0.4cm}

Now let us introduce the hypotheses concerning the amplitude function. \\

\noindent \textbf{Assumption ($\mathbf{A_{p_1,\mu}}$).} Let $\mu \in (0,1]$.\\ A function $U : (p_1, p_2] \longrightarrow \C$ satisfies Assumption (A$_{p_1,\mu}$) if and only if there exists a function $\tilde{u} : [p_1,p_2] \longrightarrow \C$ such that
	\begin{equation*}
		\forall \, p \in (p_1,p_2] \qquad U(p )= (p - p_1)^{\mu -1} \, \tilde{u}(p) \; ,
	\end{equation*}
	where $\tilde{u}$ is assumed continuous on $[p_1,p_2]$, differentiable on $(p_1,p_2)$ with $\tilde{u}' \in L^1 \big((p_1,p_2), \C\big)$, and $\tilde{u}(p_1) \neq 0$ if $\mu \neq 1$.\\ The point $p_1$ is called \emph{singularity} of $U$, and $\tilde{u}$ the \emph{regular part} of $U$.\\
		
\noindent According to this assumption, the amplitude is singular at the left endpoint of the interval; we choose this position only for simplicity. The strength of the singularity is described by the value of $\mu-1$.\\

Now let us state the first van der Corput type estimate of the considered integrals \eqref{oscillatory}. Here we suppose that the phase function $\psi$ has a stationary point $p_0$ of order $\rho$ which belongs to the integration interval. The furnished estimate is uniform with respect to the position of $p_0$; an upper bound of the constant is given in terms of the regular part $\tilde{u}$ of the amplitude and the non-degenerate par $\tilde{\psi}$ of the phase function.\\
To prove this first result, we adapt the method employed by E. Stein \cite{stein}. More precisely, we decompose the integration interval: away from the stationary point $p_0$ and the singularity $p_1$, we integrate by parts to create a factor exhibiting the decay. Then we couple the distance to the singular points with the large parameter $\omega$ to obtain the final decay rate. Thanks to this coupling, the integral on the small intervals containing the singular points are estimated against the length of this domain.

\begin{VDC} \label{VDC}
	Let $\rho > 1$, $\mu \in (0,1]$ and choose $p_0 \in [p_1,p_2]$. Suppose that the functions $\psi : I \longrightarrow \R$ and $U : (p_1, p_2] \longrightarrow \C$ satisfy Assumption \emph{(P$_{p_0,\rho}$)} and Assumption \emph{(A$_{p_1,\mu}$)}, respectively. Moreover suppose that $\psi'$ is monotone on $I_{p_0}^-$ and $I_{p_0}^+$, where
	\begin{equation*}
		I_{p_0}^- := \left\{ p \in I \, \big| \, p \leqslant p_0 \right\} \qquad , \qquad I_{p_0}^+ := \left\{ p \in I \, \big| \, p \geqslant p_0 \right\} \; .
	\end{equation*}
	Then
	\begin{equation*}
		\left| \int_{p_1}^{p_2} U(p) \, e^{i \omega \psi(p)} \, dp \right| \leqslant C(U,\psi) \, \omega^{-\frac{\mu}{\rho}} \; ,
	\end{equation*}
	for all $\omega > 0$, where the constant $C(U,\psi) > 0$ is given by
	\begin{equation*}
		C(U,\psi) := \frac{3}{\mu} \, \left\| \tilde{u} \right\|_{L^{\infty}(p_1,p_2)} + \left( 8 \left\| \tilde{u} \right\|_{L^{\infty}(p_1,p_2)} + 2 \left\| \tilde{u}' \right\|_{L^1(p_1,p_2)} \right) \left(\min_{p \in [p_1,p_2]} \left| \tilde{\psi}(p) \right| \right)^{-1} \; .
	\end{equation*}
\end{VDC}

\noindent Before proving this theorem, let us illustrate the monotonicity hypothesis of $\psi'$ with the examples given in \ref{EXPLE}.

\begin{REM1}
	\em \begin{enumerate}
		\item In the setting of Examples \ref{EXPLE} i), if $\big| \psi^{(N)} \big| > 0$ on $I$, then $\psi'$ is monotone on both intervals $I_{p_0}^-$ and $I_{p_0}^+$.\\
	Indeed if $N =2$, then it is clear that the hypothesis $\big| \psi'' \big| > 0$ implies the result. Suppose now that $N \geqslant 3$; then applying Taylor's formula to $\psi''$, namely
	\begin{equation*}
		\psi''(p) = \frac{1}{(N-3)!} \int_{p_0}^p (p-x)^{N-3} \, \psi^{(N)}(x) \, dx \; ,
	\end{equation*}
	for all $p \in I$, we observe that $\psi''$ has a constant sign on $I_{p_0}^-$ and $I_{p_0}^+$, which provides the result.
		\item In the setting of Examples \ref{EXPLE} ii), we note that $\psi'$ is monotone on $\R_-$ and $\R_+$.
	\end{enumerate}	
\end{REM1}

\noindent In favour of the readability of the proof, we shall postpone the proofs of some technical details to \ref{Details}. In particular, the situations where the stationary point is close to the border of the domain of integration will be discussed.

\begin{proof}[Proof of Theorem \ref{VDC}]
	Let $p_0 \in (p_1,p_2)$. We shall study the cases $p_0 = p_1$ and $p_0 = p_2$ at the end of the proof. Let $\omega > 0$, choose $\delta > 0$ sufficiently small\footnote[1]{See Proofs of technical details \ref{Details}} such that the following splitting of the integral is well-defined:
	\begin{align}
		\int_{p_1}^{p_2} U(p) \, e^{i \omega \psi(p)} \, dp	& \label{splitting} = \int_{p_1}^{p_1 + \delta} \dots \quad + \int_{p_1 + \delta}^{p_0-\delta} \dots \quad + \int_{p_0-\delta}^{p_0 + \delta} \dots \quad + \int_{p_0 + \delta}^{p_2} \dots \\
															& =: I^{(1)}(\omega) + I^{(2)}(\omega) + I^{(3)}(\omega) + I^{(4)}(\omega) \; . \nonumber
	\end{align}
	Let us estimate each integral.
	\begin{itemize}
		\item \textit{Study of $I^{(1)}(\omega)$.} We use the smallness of the interval to estimate this integral:
		\begin{equation*}
			\Big| I^{(1)}(\omega) \Big| \leqslant \int_{p_1}^{p_1 + \delta} \left| U(p) \right| dp \leqslant \left\| \tilde{u} \right\|_{L^{\infty}(p_1,p_2)} \int_{p_1}^{p_1 + \delta} (p-p_1)^{\mu-1} \, dp = \frac{\left\| \tilde{u} \right\|_{L^{\infty}(p_1,p_2)}}{\mu} \, \delta^{\mu} \; .
		\end{equation*}
		\item \textit{Study of $I^{(2)}(\omega)$.} Here we use the oscillations of the integrand to obtain an estimate. We shall suppose that $\tilde{\psi}$ is positive on $\{ p \in I \, | \, p < p_0 \}$, which implies the non-negativity of $\psi'$; the other case can be studied in the same manner.  Since $\psi'$ does not vanish on $[p_1 + \delta, p_0 - \delta]$, the substitution $s = \psi(p)$ can be employed. Setting $\varphi := \psi^{-1}$, $s_1 := \psi(p_1 + \delta)$ and $s_2 := \psi(p_0 - \delta)$, we obtain
		\begin{align*}
			I^{(2)}(\omega)	& = \int_{s_1}^{s_2} U\big( \varphi(s) \big) \, \varphi'(s) \, e^{i \omega s} \, ds \\
							& = (i \omega)^{-1} \bigg( \Big[ (U \circ \varphi)(s) \,  \varphi'(s) \, e^{i \omega s} \Big]_{s_1}^{s_2} - \int_{s_1}^{s_2} \big( (U \circ \varphi) \, \varphi' \big)'(s) \, e^{i \omega s} \, ds \bigg) \; ;
		\end{align*}
		the last equality was obtained by integrating by parts.\\
		Let us control the boundary terms and the integral. Firstly, we have
		\begin{equation} \label{estU}
			\big| U(p) \big| \leqslant \delta^{\mu-1} \left\| \tilde{u} \right\|_{L^{\infty}(p_1+\delta,p_0-\delta)} \leqslant \delta^{\mu-1} \left\| \tilde{u} \right\|_{L^{\infty}(p_1,p_2)} \; ,
		\end{equation}
		for all $p \in [p_1 + \delta, p_0 - \delta]$, since $U(p) = (p-p_1)^{\mu-1} \tilde{u}(p)$ by hypothesis. Moreover the assumption on $\psi'$ implies
		\begin{equation*}
			\forall \, p \in [p_1 + \delta, p_0 - \delta] \qquad \left| \psi'(p) \right| \geqslant \delta^{\rho-1} \, m \; ,
		\end{equation*}
		where $\displaystyle m := \min_{p \in [p_1,p_2]} \left| \tilde{\psi}(p) \right| > 0$. Combining this with the definition of $\varphi$ leads to
		\begin{equation} \label{estP}
			\forall \, s \in \left[s_1,s_2\right] \qquad \left| \varphi'(s) \right| \leqslant \delta^{1-\rho} \, m^{-1} \; .
		\end{equation}
		Inequalities \eqref{estU} and \eqref{estP} permit to estimate the boundary terms:
		\begin{equation*}
			\bigg| \Big[ (U \circ \varphi)(s) \,  \varphi'(s) \, e^{i \omega s} \Big]_{s_1}^{s_2} \, \bigg| \leqslant 2 \, \left\| \tilde{u} \right\|_{L^{\infty}(p_1,p_2)} \, m^{-1} \, \delta^{\mu-\rho} \; .
		\end{equation*}
		It remains to control the integral. We have
		\begin{equation*}
			\big( (U \circ \varphi) \, \varphi' \big)' = (U' \circ \varphi) \, \left( \varphi' \right)^{\, 2} + (U \circ \varphi) \, \varphi'' \; ,
		\end{equation*}
		by the product rule; consequently,
		\begin{align}
			\bigg| \int_{s_1}^{s_2} \big( (U \circ \varphi) \, \varphi' \big)'(s) \, e^{i \omega s} \, ds \bigg|	& \leqslant \int_{s_1}^{s_2} \Big|(U' \circ \varphi)(s) \, \varphi'(s)^2 \Big| \, ds \nonumber \\
													& \qquad \qquad + \; \int_{s_1}^{s_2} \Big| (U \circ \varphi)(s) \, \varphi''(s) \Big| \, ds \nonumber \\
													& \leqslant \int_{s_1}^{s_2} \Big|(U' \circ \varphi)(s) \, \varphi'(s) \Big| \, ds \; \delta^{1-\rho} \, m^{-1} \nonumber \\
													& \qquad \qquad + \; \left\| U \right\|_{L^{\infty}(p_1+\delta,p_0-\delta)} \int_{s_1}^{s_2} \big| \varphi''(s) \big| \, ds \nonumber \\
													& \leqslant \int_{p_1+\delta}^{p_0-\delta} \big| U'(p) \big| \, dp \; \delta^{1-\rho} \, m^{-1} \nonumber \\
													& \label{estprincipal} \qquad \qquad + \; \delta^{\mu-1} \left\| \tilde{u} \right\|_{L^{\infty}(p_1,p_2)} \int_{s_1}^{s_2} \big| \varphi''(s) \big| \, ds \; .
		\end{align}
		The definition of $U$ implies
		\begin{align}
			\int_{p_1+\delta}^{p_0-\delta} \big| U'(p) \big| \, dp	& \leqslant \int_{p_1+\delta}^{p_0-\delta} \Big| (\mu-1) (p-p_1)^{\mu-2} \, \tilde{u}(p) \Big| \, dp \nonumber \\
	& \qquad \qquad + \; \int_{p_1+\delta}^{p_0-\delta} \Big| (p-p_1)^{\mu-1} \, \tilde{u}'(p) \Big| \, dp \nonumber \\
	& \leqslant \int_{p_1+\delta}^{p_0-\delta} (1-\mu) (p-p_1)^{\mu-2} \, dp \; \left\| \tilde{u} \right\|_{L^{\infty}(p_1,p_2)} \nonumber \\
	& \qquad \qquad + \; \delta^{\mu-1} \int_{p_1+\delta}^{p_0-\delta} \big| \tilde{u}'(p) \big| \, dp \nonumber \\
	& \label{est1} \leqslant \delta^{\mu-1} \left\| \tilde{u} \right\|_{L^{\infty}(p_1,p_2)} + \delta^{\mu-1} \left\| \tilde{u}' \right\|_{L^1(p_1,p_2)} \; ;
		\end{align}
		the last inequality was obtained employing the fact that
		\begin{equation*}
			\int_{p_1+\delta}^{p_0-\delta} (1-\mu) (p-p_1)^{\mu-2} \, dp = \delta^{\mu-1} - (p_0 - \delta - p_1)^{\mu-1} \leqslant \delta^{\mu-1} \; .
		\end{equation*}
		Moreover the relation $\displaystyle \varphi'' = \left( \frac{-\psi''}{\psi'^{\, 3}} \right) \circ \varphi$ furnishes the following equalities,
		\begin{equation*}
			\int_{s_1}^{s_2} \big| \varphi''(s) \big| \, ds = \int_{s_1}^{s_2} \left| \frac{-\psi''\big( \varphi(s)\big)}{\psi'\big( \varphi(s)\big)^3} \right| ds = \int_{p_1 + \delta}^{p_0-\delta} \left| \frac{-\psi''(p)}{\psi'(p)^2} \right| dp = \left| \int_{p_1 + \delta}^{p_0-\delta} \frac{-\psi''(p)}{\psi'(p)^2} \, dp \right| \; ,
		\end{equation*}
		the last equality comes from the fact that $\psi'$ is monotone on $[p_1,p_0]$ and so $\psi''$ has a constant sign on $[p_1+\delta, p_0-\delta]$. Then
		\begin{equation} \label{est2}
			\int_{s_1}^{s_2} \big| \varphi''(s) \big| \, ds = \left| \int_{p_1 + \delta}^{p_0-\delta} \left( \frac{1}{\psi'} \right)'(p) \, dp \right| = \left| \frac{1}{\psi'(p_0 - \delta)} - \frac{1}{\psi'(p_1 + \delta)} \right| \leqslant \delta^{1-\rho} m^{-1} \; ,
		\end{equation}
		where we used $| \psi'(p) | \geqslant \delta^{\rho-1} \, m$, for $p \in [p_1+\delta, p_0-\delta]$. Putting \eqref{est1} and \eqref{est2} in \eqref{estprincipal} provides
		\begin{equation*}
			\bigg| \int_{s_1}^{s_2} \big( (U \circ \varphi) \, \varphi' \big)'(s) \, e^{i \omega s} \, ds \bigg| \leqslant \left( 2 \left\| \tilde{u} \right\|_{L^{\infty}(p_1,p_2)} + \left\| \tilde{u}' \right\|_{L^1(p_1,p_2)} \right) m^{-1} \, \delta^{\mu-\rho} \; .
		\end{equation*}
		We are now able to estimate $I^{(2)}(\omega)$:
		\begin{equation*}
			\Big| I^{(2)}(\omega) \Big| \leqslant \left( 4 \left\| \tilde{u} \right\|_{L^{\infty}(p_1,p_2)} + \left\| \tilde{u}' \right\|_{L^1(p_1,p_2)} \right) m^{-1} \, \delta^{\mu-\rho} \, \omega^{-1} \; .
		\end{equation*}
		\item \textit{Study of $I^{(3)}(\omega)$.} The small length of the interval is used once again to estimate this integral:
		\begin{equation*}
			\Big| I^{(3)}(\omega) \Big| \leqslant \frac{\left\| \tilde{u} \right\|_{L^{\infty}(p_1,p_2)}}{\mu} \, \big( (p_0 + \delta -p_1)^{\mu} - (p_0 - \delta - p_1)^{\mu} \big) \leqslant 2 \, \frac{\left\| \tilde{u} \right\|_{L^{\infty}(p_1,p_2)}}{\mu} \, \delta^{\mu} \; ;
		\end{equation*}
		Lemma \ref{LEM} was employed to obtain the last inequality.
		\item \textit{Study of $I^{(4)}(\omega)$.} On $[p_0 + \delta, p_2]$, we can bound from below the absolute value of the first derivative of the phase function as follows,
		\begin{equation*}
			\big| \psi' \big| \geqslant \delta^{\rho-1} \, \min_{p \in [p_1,p_2]} \left| \tilde{\psi}(p) \right| = \delta^{\rho-1} \, m \; ,
		\end{equation*}
		and we have
		\begin{equation*}
			\forall \, p \in [p_0 + \delta, p_2] \qquad (p-p_1)^{\mu-1} \leqslant (p_0 + \delta - p_1)^{\mu-1} \leqslant \delta^{\mu-1} \; .
		\end{equation*}
		Following the lines of the study of $I^{(2)}(\omega)$ and using the two previous estimates, we obtain
		\begin{equation*}
			\Big| I^{(4)}(\omega) \Big| \leqslant \left( 4 \left\| \tilde{u} \right\|_{L^{\infty}(p_1,p_2)} + \left\| \tilde{u}' \right\|_{L^1(p_1,p_2)} \right) m^{-1} \, \delta^{\mu - \rho} \, \omega^{-1} \; .
		\end{equation*}
	\end{itemize}
	To conclude the proof, we set\footnote[2]{See Proofs of technical details \ref{Details}}
	\begin{equation*}
		\delta = \omega^{-\frac{1}{\rho}} \; ,
	\end{equation*}
	which furnishes the desired estimate
	\begin{align}
		\left| \int_{p_1}^{p_2} U(p) \, e^{i \omega \psi(p)} \, dp \right|	& \leqslant \left| I^{(1)}(\omega) \right| + \left| I^{(2)}(\omega) \right| + \left| I^{(3)}(\omega) \right| + \left| I^{(4)}(\omega) \right| \nonumber \\
		& \leqslant \frac{\left\| \tilde{u} \right\|_{L^{\infty}(p_1,p_2)}}{\mu} \, \omega^{-\frac{\mu}{\rho}} \; + \; 2 \, \frac{\left\| \tilde{u} \right\|_{L^{\infty}(p_1,p_2)}}{\mu} \, \omega^{-\frac{\mu}{\rho}} \nonumber \\
		& \qquad \qquad + \; 2 \left( 4 \left\| \tilde{u} \right\|_{L^{\infty}(p_1,p_2)} + \left\| \tilde{u}' \right\|_{L^1(p_1,p_2)} \right) m^{-1} \, \omega^{-\frac{\mu-\rho}{\rho}} \, \omega^{-1} \nonumber \\
		& =: C(U,\psi) \, \omega^{-\frac{\mu}{\rho}} \; , \label{est0}
	\end{align}
	where
	\begin{equation*}
		C(U,\psi) := \frac{3}{\mu} \, \left\| \tilde{u} \right\|_{L^{\infty}(p_1,p_2)} + \left( 8 \left\| \tilde{u} \right\|_{L^{\infty}(p_1,p_2)} + 2 \left\| \tilde{u}' \right\|_{L^1(p_1,p_2)} \right) m^{-1} \; .
	\end{equation*}
	And since the right-hand side of \eqref{est0} does not depend on $p_0$, then the estimate holds also for $p_0 = p_1$ and $p_0 = p_2$.
\end{proof}

\noindent Let us give the details of two points of the above proof.

\begin{Details} \label{Details}\em
	\begin{enumerate}
		\item We remark that if $\displaystyle \delta \geqslant \min \left\{ \frac{p_0 - p_1}{2}, p_2 - p_0 \right\}$ then the splitting \eqref{splitting} of the integral is not well-defined. Consequently we proved the result only in the case of small $\delta$, that is to say in the case of large $\omega$ since $\delta = \omega^{-\frac{1}{\rho}}$.\\ But one can establish the desired estimate in the case of large $\delta$ (i.e. small $\omega$) by adapting slightly the above method. If $\displaystyle \delta \geqslant \frac{p_0 - p_1}{2}$ then
		\begin{equation} \label{Int1}
			\left| \int_{p_1}^{p_0} U(p) \, e^{i \omega \psi(p)} \, dp \right| \leqslant \frac{\| \tilde{u} \|_{L^{\infty}(p_1,p_2)}}{\mu} \, (p_0 - p_1)^{\mu} \leqslant 2 \, \frac{\| \tilde{u} \|_{L^{\infty}(p_1,p_2)}}{\mu} \, \delta^{\mu} \; ;
		\end{equation}
		in the same way, if $\delta \geqslant p_2 - p_0$ then
		\begin{equation} \label{Int2}
			\left| \int_{p_0}^{p_2} U(p) \, e^{i \omega \psi(p)} \, dp \right| \leqslant\frac{\| \tilde{u} \|_{L^{\infty}(p_1,p_2)}}{\mu} \, \delta^{\mu} \; .
		\end{equation}
		Supposing $\displaystyle \frac{p_0 - p_1}{2} \geqslant p_2 - p_0$ without loss of generality, we have to consider three cases:
		\begin{itemize}
			\item \textit{Case $\displaystyle \delta \geqslant \frac{p_0 - p_1}{2}$.} Here the integral can be easily estimated:
			\begin{equation} \label{largedelta1}
				\left| \int_{p_1}^{p_2} U(p) \, e^{i \omega \psi(p)} \, dp \right| \leqslant \left| \int_{p_1}^{p_0} \dots \right| \, + \, \left| \int_{p_0}^{p_2} \dots \right| \leqslant 3 \, \frac{\| \tilde{u} \|_{L^{\infty}(p_1,p_2)}}{\mu} \, \delta^{\mu} \; ,
			\end{equation}
			where we used the estimates \eqref{Int1} and \eqref{Int2}.
			\item \textit{Case $\displaystyle \frac{p_0 - p_1}{2} > \delta \geqslant p_2 - p_0$.} In this case, we have
			\begin{align}
				\left| \int_{p_1}^{p_2} U(p) \, e^{i \omega \psi(p)} \, dp \right|	& \leqslant \left| \int_{p_1}^{p_1 + \delta} \dots \right| \, + \, \left| \int_{p_1 + \delta}^{p_0 - \delta} \dots \right| \, + \, \left| \int_{p_0 - \delta}^{p_0} \dots \right| \, + \, \left| \int_{p_0}^{p_2} \dots \right| \nonumber \\
				& \leqslant \frac{\| \tilde{u} \|_{L^{\infty}(p_1,p_2)}}{\mu} \, \delta^{\mu} + \left( 4 \left\| \tilde{u} \right\|_{L^{\infty}(p_1,p_2)} + \left\| \tilde{u}' \right\|_{L^1(p_1,p_2)} \right) m^{-1} \, \delta^{\mu-\rho} \, \omega^{-1} \nonumber \\
				& \label{largedelta2} \qquad + \frac{\| \tilde{u} \|_{L^{\infty}(p_1,p_2)}}{\mu} \, \delta^{\mu} + \frac{\| \tilde{u} \|_{L^{\infty}(p_1,p_2)}}{\mu} \, \delta^{\mu} \; ;
			\end{align}
			The three first integrals were estimated using the proof of Theorem \ref{VDC}, whereas we employed \eqref{Int2} to bound from above the last integral.
			\item \textit{Case $p_2 - p_0 > \delta$.} This is the setting of the proof of Theorem \ref{VDC}.			
		\end{itemize}
		Replacing $\delta$ by $\omega^{-\frac{1}{\rho}}$ in \eqref{largedelta1} and \eqref{largedelta2}, we observe that the oscillatory integral is always bounded from above by $C(U,\psi) \, \omega^{-\frac{\mu}{\rho}}$ in the three cases.\\ This remark holds also for the other results of this section: the estimates are true for any $\omega > 0$, even if we suppose $\delta > 0$ sufficiently small (i.e. $\omega$ sufficiently large) in the proofs.
		\item Let us explain the choice $\delta = \omega^{-\frac{1}{\rho}}$, following the idea of \cite[p. 197-198]{zygmund}. Formally, before replacing $\delta$ by $\omega^{-\frac{1}{\rho}}$ in the proof, we obtain an estimate of the form
		\begin{equation*}
			\left| \int_{p_1}^{p_2} U(p) \, e^{i \omega \psi(p)} \, dp \right| \leqslant f_{\omega}(\delta) \; ,
		\end{equation*}
		where $f_{\omega}(\delta) := c_1 \delta^{\mu} + c_2 \, \omega^{-1} \delta^{\mu-\rho}$, for certain constants $c_1 , c_2 > 0$. We note that $(f_{\omega})'$ vanishes at a unique point $\delta_0$ defined by
		\begin{equation*}
			\delta_0 := \left( \frac{\mu}{\rho-\mu} \, \frac{c_1}{c_2} \right)^{-\frac{1}{\rho}} \omega^{-\frac{1}{\rho}} \; .
		\end{equation*}
		Since $\displaystyle \lim_{\delta \rightarrow 0^+} f_{\omega}(\delta) = \lim_{\delta \rightarrow + \infty} f_{\omega}(\delta) = + \infty$, $\delta_0$ is then the minimum of $f_{\omega}$. Therefore the choice $\delta = \omega^{-\frac{1}{\rho}}$ seems to be optimal regarding the decay rate. However we don't choose exactly the minimum of $f_{\omega}$ for simplicity. Hence the constant $C(U,\psi)$ can not be optimal.
	\end{enumerate}
\end{Details}

\begin{REM2} \label{REM2} \em
	As explained above, we do not furnish an optimal constant. Nevertheless it could be slightly improved in certain special cases, as for example in the case of regular amplitude, namely $\mu = 1$ with $U = \tilde{u}$. Indeed, the study of $I^{(1)}(\omega)$ is not necessary in this situation and the employed computations to establish \eqref{est1} are not needed, since we have
	\begin{equation*}
		\int_{p_1}^{p_0-\delta} \big| U'(p) \big| dp \leqslant \left\| \tilde{u}' \right\|_{L^1(p_1,p_2)} \qquad , \qquad \int_{p_0+\delta}^{p_2} \big| U'(p) \big| dp \leqslant \left\| \tilde{u}' \right\|_{L^1(p_1,p_2)} \; .
	\end{equation*}
	It follows that we can estimate $I^{(2)}(\omega)$ and $I^{(4)}(\omega)$ more precisely, namely,
	\begin{equation} \label{I2}
		\Big| I^{(j)}(\omega) \Big| \leqslant \left( 3 \left\| \tilde{u} \right\|_{L^{\infty}(p_1,p_2)} + \left\| \tilde{u}' \right\|_{L^1(p_1,p_2)} \right) m^{-1} \, \delta^{1 - \rho} \, \omega^{-1} \; ,
	\end{equation}
	with $j=2,4$, leading to
	\begin{equation*}
		C(U,\psi) := 2 \left\| \tilde{u} \right\|_{L^{\infty}(p_1,p_2)} + \left( 6 \left\| \tilde{u} \right\|_{L^{\infty}(p_1,p_2)} + 2 \left\| \tilde{u}' \right\|_{L^1(p_1,p_2)} \right) m^{-1} \; .
	\end{equation*}
	This refined constant will be employed several times in Section 2.
\end{REM2}

\vspace{0.4cm}

In the second result, we assume that the stationary point $p_0$ is outside the domain of integration $[p_1,p_2]$. In this case, the problem is that the derivative of the phase function does not vanish on the integration interval but can be arbitrarily close to $0$ if the stationary point is close to this interval. In order to obtain an estimate which is uniform with respect to the position of $p_0$, we apply the same frequency decomposition of the domain of integration as above. We obtain a uniform estimate with the same decay rate as obtained in the first result.

\begin{VDC3} \label{VDC3}
	Let $\rho > 1$, $\mu \in (0,1]$ and choose $p_0 \in I \backslash [p_1,p_2]$. Suppose that the functions $\psi : I \longrightarrow \R$ and $U : (p_1, p_2] \longrightarrow \C$ satisfy Assumption \emph{(P$_{p_0,\rho}$)} and Assumption \emph{(A$_{p_1,\mu}$)}, respectively. Moreover suppose that  $\psi'$ is monotone on $[p_1,p_2]$. Then
	\begin{equation*}
		\left| \int_{p_1}^{p_2} U(p) \, e^{i \omega \psi(p)} \, dp \right| \leqslant \tilde{C}(U,\psi) \, \omega^{-\frac{\mu}{\rho}} \; ,
	\end{equation*}
	for all $\omega > 0$, where the constant $\tilde{C}(U,\psi) > 0$ is given by
	\begin{equation*}
		\tilde{C}(U,\psi) := \frac{2}{\mu} \, \left\| \tilde{u} \right\|_{L^{\infty}(p_1,p_2)} + \left( 4 \left\| \tilde{u} \right\|_{L^{\infty}(p_1,p_2)} + \left\| \tilde{u}' \right\|_{L^1(p_1,p_2)} \right) \left(\min_{p \in [p_1,p_2]} \left| \tilde{\psi}(p) \right| \right)^{-1} \; .
	\end{equation*}
\end{VDC3}	

\begin{proof}
	 Let $\omega > 0$ and $\delta > 0$ sufficiently small. We split the integral,
	\begin{align*}
		\int_{p_1}^{p_2} U(p) \, e^{i \omega \psi(p)} \, dp	& = \int_{p_1}^{p_1 + \delta} \dots \quad + \int_{p_1 + \delta}^{p_2-\delta} \dots \quad + \int_{p_2-\delta}^{p_2} \dots \\
															& =: I^{(1)}(\omega) + I^{(2)}(\omega) + I^{(3)}(\omega) \; ,
	\end{align*}
	where $I^{(1)}(\omega)$ and $I^{(3)}(\omega)$ are bounded from above by $\displaystyle \frac{\left\| \tilde{u} \right\|_{L^{\infty}(p_1,p_2)}}{\mu} \, \delta^{\mu}$, using the smallness of the interval. The study of $I^{(2)}(\omega)$ is based on the employed method of the second point of the proof of Theorem \ref{VDC}, which provides
	\begin{equation*}
		\Big| I^{(2)}(\omega) \Big| \leqslant \left( 4 \left\| \tilde{u} \right\|_{L^{\infty}(p_1,p_2)} + \left\| \tilde{u}' \right\|_{L^1(p_1,p_2)} \right) m^{-1} \, \delta^{\mu - \rho} \, \omega^{-1} \; ;
	\end{equation*}
	we used the fact that
	\begin{equation*}
		\forall \, p \in [p_1 + \delta, p_2 - \delta] \qquad | U(p) | \leqslant \delta^{\mu-1} \| \tilde{u} \|_{L^{\infty}(p_1,p_2)} \; ,
	\end{equation*}
	and
	\begin{equation*}
		\forall \, p \in [p_1 + \delta, p_2 - \delta] \qquad \big| \psi'(p) \big| \geqslant \left\{ \begin{array}{rl}
			& \hspace{-3mm} (p_1 + \delta - p_0)^{\rho-1} \, m \geqslant \delta^{\rho-1} \, m \; , \quad \text{if } p_0 < p_1 \; , \\ [2mm]
			& \hspace{-3mm} (p_0 - p_2 + \delta)^{\rho-1} \, m \geqslant \delta^{\rho-1} \, m \; , \quad \text{if } p_0 > p_2 \; ,
	\end{array} \right.
	\end{equation*}
	with $\displaystyle m := \min_{p \in [p_1,p_2]} \left| \tilde{\psi} (p) \right|$. Finally we set $\delta = \omega^{-\frac{1}{\rho}}$ to conclude.
\end{proof}

\begin{VDC3bis} \label{VDC3bis} \em
	In the case of regular amplitude, one can use the refined the estimate \eqref{I2} of $I^{(2)}(\omega)$ provided in Remark \ref{REM2}. Hence the constant $\tilde{C}(U,\psi)$ becomes in this situation,
	\begin{equation*}
		\tilde{C}(U,\psi) := 2 \left\| \tilde{u} \right\|_{L^{\infty}(p_1,p_2)} + \left( 3 \left\| \tilde{u} \right\|_{L^{\infty}(p_1,p_2)} + \left\| \tilde{u}' \right\|_{L^1(p_1,p_2)} \right) \left(\min_{p \in [p_1,p_2]} \left| \tilde{\psi}(p) \right| \right)^{-1} \; .
	\end{equation*}
\end{VDC3bis}

\vspace{0.4cm}

We derive from the two previous theorems the following corollary, furnishing an estimate of the oscillatory integral which does not depend on the position of the stationary point.

\begin{COR1} \label{COR1}
	Let $\rho > 1$, $\mu \in (0,1]$ and choose $p_0 \in I$. Suppose that the functions $\psi : I \longrightarrow \R$ and $U : (p_1, p_2] \longrightarrow \C$ satisfy Assumption \emph{(P$_{p_0,\rho}$)} and Assumption \emph{(A$_{p_1,\mu}$)}, respectively. Moreover suppose that $\psi'$ is monotone on $I_{p_0}^-$ and $I_{p_0}^+$, where
	\begin{equation*}
		I_{p_0}^- := \left\{ p \in I \, \big| \, p \leqslant p_0 \right\} \qquad , \qquad I_{p_0}^+ := \left\{ p \in I \, \big| \, p \geqslant p_0 \right\} \; .
	\end{equation*}
	Then
	\begin{equation*}
		\left| \int_{p_1}^{p_2} U(p) \, e^{i \omega \psi(p)} \, dp \, \right| \leqslant C(U,\psi) \, \omega^{-\frac{\mu}{\rho}} \; ,
	\end{equation*}
	for all $\omega > 0$, where the constant $C(U,\psi) > 0$ is given by
	\begin{equation*}
		C(U,\psi) := \frac{3}{\mu} \, \left\| \tilde{u} \right\|_{L^{\infty}(p_1,p_2)} + \left( 8 \left\| \tilde{u} \right\|_{L^{\infty}(p_1,p_2)} + 2 \left\| \tilde{u}' \right\|_{L^1(p_1,p_2)} \right) \left(\min_{p \in [p_1,p_2]} \left| \tilde{\psi}(p) \right| \right)^{-1} \; .
	\end{equation*}
\end{COR1}

\begin{proof}
	The result is a direct consequence of Theorems \ref{VDC} and \ref{VDC3}. Let us distinguish two cases.
	\begin{itemize}
		\item \textit{Case $p_0 \in [p_1, p_2]$}. This corresponds to the setting of Theorem \ref{VDC}. So
		\begin{equation*}
			\left| \int_{p_1}^{p_2} U(p) \, e^{i \omega \psi(p)} \, dp \, \right| \leqslant C(U,\psi) \, \omega^{-\frac{\mu}{\rho}} \; ,
		\end{equation*}
		with
		\begin{equation*}
			C(U,\psi) = \frac{3}{\mu} \, \left\| \tilde{u} \right\|_{L^{\infty}(p_1,p_2)} + \left( 8 \left\| \tilde{u} \right\|_{L^{\infty}(p_1,p_2)} + 2 \left\| \tilde{u}' \right\|_{L^1(p_1,p_2)} \right) \left(\min_{p \in [p_1,p_2]} \left| \tilde{\psi}(p) \right| \right)^{-1} \; .
		\end{equation*}
		\item \textit{Case $p_0 \notin [p_1, p_2]$}. In this case, either $[p_1,p_2] \subset I_{p_0}^-$ or $[p_1,p_2] \subset I_{p_0}^+$. Since $\psi'$ is assumed monotone on both intervals $I_{p_0}^-$ and $I_{p_0}^+$, Theorem \ref{VDC3} is applicable and furnishes
		\begin{equation*}
			\left| \int_{p_1}^{p_2} U(p) \, e^{i \omega \psi(p)} \, dp \, \right| \leqslant \tilde{C}(U,\psi) \, \omega^{-\frac{\mu}{\rho}} \; ,
		\end{equation*}
		with $\displaystyle \tilde{C}(U,\psi) := \frac{2}{\mu} \, \left\| \tilde{u} \right\|_{L^{\infty}(p_1,p_2)} + \left( 4 \left\| \tilde{u} \right\|_{L^{\infty}(p_1,p_2)} + \left\| \tilde{u}' \right\|_{L^1(p_1,p_2)} \right) \left(\min_{p \in [p_1,p_2]} \left| \tilde{\psi}(p) \right| \right)^{-1} \; .$		
	\end{itemize}
	Now we note that $\tilde{C}(U,\psi) \leqslant C(U,\psi)$. So for any $p_0 \in I$, we have the desired uniform estimate, namely,
	\begin{equation*}
		\forall \, \omega > 0 \qquad \left| \int_{p_1}^{p_2} U(p) \, e^{i \omega \psi(p)} \, dp \, \right| \leqslant C(U,\psi) \, \omega^{-\frac{\mu}{\rho}} \; .
	\end{equation*}
\end{proof}

\begin{REM2bis} \label{REM2bis} \em
	As previously, we furnish a better constant in the case of regular amplitude:
	\begin{equation*}
		C(U,\psi) := 2 \left\| \tilde{u} \right\|_{L^{\infty}(p_1,p_2)} + \left( 6 \left\| \tilde{u} \right\|_{L^{\infty}(p_1,p_2)} + 2 \left\| \tilde{u}' \right\|_{L^1(p_1,p_2)} \right) \left(\min_{p \in [p_1,p_2]} \left| \tilde{\psi}(p) \right| \right)^{-1} \; .
	\end{equation*}
\end{REM2bis}

\vspace{0.4cm}

In the last theorem of this section, we give up the uniformity requirement in favour of an improved decay rate. To do so, we consider the weaker hypothesis that the derivative of the phase function is non-zero inside $[p_1,p_2]$. This necessitates to employ the quantity $\displaystyle \min_{[p_1,p_2]} |\psi'|$ as a lower-bound of the derivative of the phase, leading to an improved decay but which is uniform only if $\psi$ has no stationary points at all. Otherwise the constant depends intrinsically on the distance between the stationary point $p_0$ and $[p_1,p_2]$: the more the stationary point approaches the integration interval, the larger the constant is. Finally let us remark that this new estimate will be helpful to establish some results of the next section.

\begin{VDC4} \label{VDC4}
	Let $\mu \in (0,1]$. Suppose that the function $U : (p_1,p_2] \longrightarrow \C$ satisfies Assumption \emph{(A$_{p_1,\mu}$)}. Moreover suppose that $\psi \in \mathcal{C}^2(I)$ such that $\psi'$ does not vanish and is monotone on $[p_1,p_2]$. Then
	\begin{equation*}
		\left| \int_{p_1}^{p_2} U(p) \, e^{i \omega \psi(p)} \, dp \right| \leqslant C^c(U,\psi) \, \omega^{-\mu} \; ,
	\end{equation*}
	for all $\omega > 0$, where the constant $C^c(U,\psi) > 0$ is given by
	\begin{equation*}
		C^c(U,\psi) := \frac{1}{\mu} \, \left\| \tilde{u} \right\|_{L^{\infty}(p_1,p_2)} + \left( 4 \left\| \tilde{u} \right\|_{L^{\infty}(p_1,p_2)} + \left\| \tilde{u}' \right\|_{L^1(p_1,p_2)} \right) \left( \min_{p \in [p_1,p_2]} \big| \psi'(p) \big| \right)^{-1} \; .
	\end{equation*}
\end{VDC4}

\begin{proof}
	Let $\omega > 0$, $\delta > 0$ sufficiently small and split the integral as follows,
	\begin{align*}
		\int_{p_1}^{p_2} U(p) \, e^{i \omega \psi(p)} \, dp	& = \int_{p_1}^{p_1 + \delta} \dots \quad + \int_{p_1 + \delta}^{p_2} \dots \\
															& =: I^{(1)}(\omega) + I^{(2)}(\omega) \; .
	\end{align*}
	The integral $I^{(1)}(\omega)$ is bounded by $\displaystyle \frac{\left\| \tilde{u} \right\|_{L^{\infty}(p_1,p_2)}}{\mu} \, \delta^{\mu}$. Then we use the method of the second point of the proof of Theorem \ref{VDC} to obtain an estimate of $I^{(2)}(\omega)$, since $\psi'$ does not vanish on $[p_1,p_2]$. But here, we bound $|\psi'|$ from below by $\displaystyle \min_{p \in [p_1,p_2]} | \psi'(p) | =: m > 0$, leading to
	\begin{equation*}
		\Big| I^{(2)}(\omega) \Big| \leqslant \left( 4 \left\| \tilde{u} \right\|_{L^{\infty}(p_1,p_2)} + \left\| \tilde{u}' \right\|_{L^1(p_1,p_2)} \right) m^{-1} \, \delta^{\mu - 1} \omega^{-1} \; ,
	\end{equation*}
	We put $\delta = \omega^{-1}$ to finish.
\end{proof}

\begin{VDC4bis} \label{VDC4bis} \em
	\begin{enumerate}
		\item Theorem \ref{VDC4} together with Theorem \ref{VDC3} show that getting a better decay rate balances out with a loss of uniformity of the constant. Actually on the one hand, the decay rate $\omega^{-\frac{\mu}{\rho}}$ is slower than $\omega^{-\mu}$; on the other hand, $\tilde{C}(U,\psi)$ (see Theorem \ref{VDC3}) is a constant which does not depend on $p_0$, whereas $C^c(U,\psi)$ may depend on the stationary point if it exists.
		\item Let us furnish a refinement of the constant $C^c(U,\psi)$ in the case of regular amplitude. Here we do not need to consider the integral $I^{(1)}(\omega)$ and according to Remark \ref{REM2}, the estimate of $I^{(2)}(\omega)$ is improvable. Then we obtain
		\begin{equation*}
			C^c(U,\psi) := \left( 3 \left\| \tilde{u} \right\|_{L^{\infty}(p_1,p_2)} + \left\| \tilde{u}' \right\|_{L^1(p_1,p_2)} \right) \left( \min_{p \in [p_1,p_2]} \big| \psi'(p) \big| \right)^{-1} \; .
		\end{equation*}
	\end{enumerate}
\end{VDC4bis}

\section{Applications to a class of dispersive equations: slow decays and concentration phenomena}

\hspace{0.5cm} In this second section, we are interested in the time asymptotic behaviour of solutions of a certain class of evolution equations on the line whose generators are Fourier multipliers. The second derivative of the symbol defining the operator is supposed to be positive and we consider initial conditions having a Fourier transform which is singular at the frequency $p_1$.

The first aim of this section is to complete in a more general setting the results of \cite{article1} and \cite{article2}. We furnish uniform estimates of the solution in arbitrary cones as well as estimates of the $L^{\infty}$-norm for $t \geqslant 1$. To do so, we write the solution as an oscillatory integral with respect to time, permitting to apply the results of Section 1. In particular, we provide the optimal $L^{\infty}$-time decay rate of the solution of the free Schrödinger equation on the line with initial data in compact frequency bands and having singular frequencies.

Then we show that the symbol of the Fourier multiplier may influence the dispersion of the solution: outside a certain cone depending only on the symbol, the decay rate of the solution is better than inside, leading to the idea that the solution tends to be concentrated in this cone.\\

Let $f : \R \longrightarrow \C$ be a function belonging to $\mathcal{C}^{\infty}(\R)$ such that all derivatives grow at most as a polynomial at infinity. We can associate with such a \textit{symbol} $f$ an operator $f(D) : \mathcal{S}(\R) \longrightarrow \mathcal{S}(\R)$ defined by
\begin{equation*}
	\forall \, x \in \R \qquad f(D) u(x) := \frac{1}{2 \pi} \int_{\R} f(p) \, \tf u(p) \, e^{i x p} \, dp = \tf^{-1} \Big( f \, \tf u \Big)(x) \; ,
\end{equation*}
where $\tf u$ is the Fourier transform of $u \in \mathcal{S}(\R)$, namely $\displaystyle \tf u(p) = \int_{\R} u(x) \, e^{- i x p} \, dx$. Since all the derivatives of the symbol $f$ grow at most as a polynomial at infinity, $f(D)$ can be extended to a map from the tempered distributions $\mathcal{S}'(\R)$ to itself.  The operator $f(D) : \mathcal{S}'(\R) \longrightarrow \mathcal{S}'(\R)$ is called a \textit{Fourier multiplier}.\\
Secondly, for such an operator, we can introduce the following evolution equation on the line,
\begin{equation*}
	\left\{ \begin{array}{rl}
			& \hspace{-2mm} \left[ i \, \partial_t - f \big(D) \right] u(t,x) = 0 \\ [2mm]
			& \hspace{-2mm} u(0,x) = u_0(x)
	\end{array} \right. \; ,
\end{equation*}
for $t > 0$ and $x \in \R$. Supposing $u_0 \in \mathcal{S}'(\R)$, this initial value problem has a unique solution in $\displaystyle \mathcal{C}^{\infty}\big( \R_+ , \mathcal{S}'(\R) \big)$, formally given by the following solution formula,
\begin{equation} \label{formula}
	u(t,x) = \frac{1}{2 \pi} \int_{\R} \tf u_0(p) \, e^{-i t f(p) + i x p} \, dp = \tf^{-1} \Big( e^{-i t f} \tf u_0 \Big)(x) \; .
\end{equation}
Throughout this section, we shall suppose that $f'' > 0$. This hypothesis implies that the phase function which appears in the solution formula has at most one stationary point of order $1$. It is possible to consider symbols leading to phases having several stationary points; the results would be more complicated but the nature of the phenomena would be unchanged. Further one might suppose that $f''$ vanishes at several points; in this case, the order of the stationary points may be larger than $1$ and hence the methods employed below have to be adapted. \\

We shall need the following definition of a space-time cone.

\begin{DEF1}
	Let $a < b$ be two real numbers (eventually infinite). We define the space-time cone $\mathfrak{C}(a, b)$ as follows:
	\begin{equation*}
		\mathfrak{C}(a, b) := \left\{ (t,x) \in (0,+\infty) \times \R \, \Big| \, a \leqslant \frac{x}{t} \leqslant b \right\} \; .
	\end{equation*}
	Furthermore the outside $\mathfrak{C}(a, b)^c$ of the cone $\mathfrak{C}(a, b)$ is given by
	\begin{equation*}
		\mathfrak{C}(a, b)^c := \big( (0,+\infty) \times \R \big) \backslash \mathfrak{C}(a, b) \; .
	\end{equation*}
\end{DEF1}

\subsection*{Frequency band and singular frequency: the influence of the initial data}

\hspace{0.5cm} In this first subsection, we study the influence of an initial data in a compact frequency band or having a singular frequency on the asymptotic behaviour of the solution.\\

In this first result, we consider an initial data in a compact frequency band $[p_1,p_2]$ where $p_1$ is a singular frequency. Hence the solution formula can be written as an oscillatory integral as in \eqref{oscillatory}. Depending on the value of the quotient $\frac{x}{t}$, the phase has a stationary point which is either in a neighbourhood of the integration interval or far from this interval. Roughly speaking, this leads to study the solution inside the space-time cone $\mathfrak{C}\big(f'(p_1),f'(p_2)\big)$ generated by the frequency band and outside. Applying the results of the first section, we obtain two estimates: the decay is slower inside this cone and is globally affected by the singular frequency $p_1$. The results are in accordance with the more precise results in \cite{article1}, which have been obtained under stronger conditions.\\

\noindent The following condition contains the assumptions of the initial data that we shall make in the first result.\\

\noindent \textbf{Condition ($\mathbf{C_{[p_1,p_2],\mu}}$).} Fix $\mu \in (0,1]$ and let $p_1, p_2$ be two finite real numbers such that $p_1 < p_2$. \\
	A tempered distribution $u_0$ satisfies Condition (C$_{[p_1,p_2],\mu}$) if and only if $\text{supp} \, \tf u_0 = [p_1,p_2]$ and $\tf u_0$ verifies Assumption (A$_{p_1,\mu}$) on $[p_1,p_2]$.

\begin{REM8}	
	\emph{It is interesting to note that $\tf u_0$ is actually an integrable function under this condition. So $\tf u_0$ belongs to $\mathcal{S}'(\R)$ and hence an initial data satisfying Condition (C$_{[p_1,p_2],\mu}$) exists as a tempered distribution. Moreover thanks to the integrability of $\tf u_0$, the solution formula \eqref{formula} is well-defined for all $x \in \R$ and $t > 0$.}
\end{REM8}

\begin{SCHRO1} \label{SCHRO1}
	 Suppose that $u_0$ satisfies Condition \emph{(C$_{[p_1,p_2],\mu}$)} and choose two finite real numbers $\tilde{p}_1, \tilde{p}_2$ such that $\displaystyle [p_1,p_2] \subsetneq [\tilde{p}_1, \tilde{p}_2] =: \tilde{I} $. Then
		\begin{equation*}
			\forall \, (t,x) \in \mathfrak{C}\big(f'(\tilde{p}_1),f'(\tilde{p}_2)\big) \qquad \big| u(t,x) \big| \leqslant c(u_0,f) \, t^{- \frac{\mu}{2}} \; ,
		\end{equation*}
		where the constant $c(u_0,f) > 0$ is given by \eqref{c_I}. Moreover
		\begin{equation*}
			\forall \, (t,x) \in \mathfrak{C}\big(f'(\tilde{p}_1),f'(\tilde{p}_2)\big)^c \qquad \big| u(t,x) \big| \leqslant c_{\tilde{I}}^c(u_0,f) \, t^{- \mu} \; ,
		\end{equation*}
		where the constant $c_{\tilde{I}}^c(u_0,f) > 0$ is given by \eqref{c_I^c}.
\end{SCHRO1}

\begin{proof}
	We consider the solution formula given by \eqref{formula} and we factorize the phase function $p \longmapsto xp - tf(p)$ by $t$, which gives
	\begin{equation*}
		\forall \, (t,x) \in (0,+\infty) \times \R \qquad u(t,x) = \int_{p_1}^{p_2} U(p) \, e^{i t \psi(p)} \, dp \; ,
	\end{equation*}
	where
	\begin{equation*}
		\left\{ \begin{array}{rl}
				& \displaystyle \forall \, p \in (p_1,p_2] \qquad U(p) := \frac{1}{2\pi}  \, \tf u_0(p) = \frac{1}{2\pi}  \, (p-p_1)^{\mu-1} \tilde{u}(p) \; , \\
				& \vspace{-0.3cm} \\
				& \displaystyle \forall \, p \in \R \qquad \psi(p) := \frac{x}{t} \, p - f(p) \; .
		\end{array} \right.
	\end{equation*}
	It is clear that $U$ verifies Assumption (A$_{p_1,\mu}$) on $[p_1,p_2]$. Moreover, we note that
	\begin{equation*}
		\psi'(p) = \frac{x}{t} - f'(p) \; .
	\end{equation*}
	But $f'' > 0$ on $\R$, so $f' : \R \longrightarrow f(\R)$ is a bijection. This implies the uniqueness of the stationary point if it exists. Now let us distinguish two cases.
	\begin{enumerate}
		\item \textit{Case $\frac{x}{t} \in f'\big(\tilde{I}\big)$.} In this case, the stationary point $p_0$ exists, belongs to $\tilde{I} := [\tilde{p}_1, \tilde{p}_2]$ and is defined by
		\begin{equation*}
			p_0 := \big(f'\big)^{-1} \left(\frac{x}{t} \right) \; .
		\end{equation*}
		Moreover $\psi''(p) = - f''(p) < 0$ which implies $\psi''(p_0) \neq 0$. According to Examples \ref{EXPLE} i), $\psi : \R \longrightarrow \R$ satisfies Assumption (P$_{p_0,2}$) with
		\begin{equation*}
			\tilde{\psi}(p) = \left\{ \begin{array}{rl}
				& \displaystyle \frac{p-p_0}{|p-p_0|} \int_0^1 -f''\big(y(p-p_0) + p_0\big) \, dy \; , \quad \text{if} \; p \neq p_0 \; , \\
				& \vspace{-0.3cm} \\
				& \displaystyle -f''(p_0) \; , \quad \text{if} \; p = p_0 \; ,
			\end{array} \right.
		\end{equation*}
		and $\big| \tilde{\psi}(p) \big| \geqslant m > 0$ for all $p \in [p_1,p_2]$, where $\displaystyle m := \min_{p \in [p_1,p_2]} f''(p) > 0$. Furthermore $\psi'$ is monotone on $\R$ since $\psi'' = - f'' < 0$. So we can apply Corollary \ref{COR1} with $\rho = 2$, which gives for all $(t,x) \in \mathfrak{C}\big(f'(\tilde{p}_1),f'(\tilde{p}_2)\big)$,
		\begin{equation*}
			\big| u(t,x) \big| = \left| \int_{p_1}^{p_2} U(p) e^{it \psi(p)} \, dp \right| \leqslant c(u_0,f) \, t^{-\frac{\mu}{2}} \; ,
		\end{equation*}
		where
		\begin{equation} \label{c_I}
			c(u_0,f) := \frac{1}{2\pi} \, \frac{3}{\mu} \, \left\| \tilde{u} \right\|_{L^{\infty}(p_1,p_2)} + \frac{1}{\pi} \left( 4 \left\| \tilde{u} \right\|_{L^{\infty}(p_1,p_2)} + \left\| \tilde{u}' \right\|_{L^1(p_1,p_2)} \right) m^{-1} \; .
		\end{equation}
		\item \textit{Case $\frac{x}{t} \notin f'\big(\tilde{I}\big)$.} Firstly, let us suppose $\frac{x}{t} > f'(\tilde{p}_2)$. Here there is no stationary point in the integration interval. More precisely, it is possible to bound $\psi'$ from below by a non-zero constant, that is to say,
		\begin{equation*}
			\forall \, p \in [p_1,p_2] \qquad \psi'(p) = \frac{x}{t} - f'(p) \geqslant f'(\tilde{p}_2) - f'(p_2) =: m_{\tilde{p}_2} > 0 \; ,
		\end{equation*}
		and $\psi'$ is still monotone. According to Theorem \ref{VDC4}, we obtain the following estimate of the solution,
		\begin{equation*}
			\forall \, t > 0 \qquad \forall \, x > f'(\tilde{p}_2) \, t \qquad \big| u(t,x) \big| \leqslant c_{x/t > f'(\tilde{p}_2)}^c(u_0,f) \, t^{-\mu} \; ,
		\end{equation*}
		with $\displaystyle c_{x/t > f'(\tilde{p}_2)}^c(u_0,f) := \frac{1}{2 \pi} \, \frac{1}{\mu} \, \left\| \tilde{u} \right\|_{L^{\infty}(p_1,p_2)} + \frac{1}{2 \pi} \left( 4 \left\| \tilde{u} \right\|_{L^{\infty}(p_1,p_2)} + \left\| \tilde{u}' \right\|_{L^1(p_1,p_2)} \right) m_{\tilde{p}_2}^{-1}$.\\ In the other case $\frac{x}{t} < f'(\tilde{p}_1)$, similar arguments furnish
		\begin{equation*}
			\forall \, t > 0 \qquad \forall \, x < f'(\tilde{p}_1) \, t \qquad \big| u(t,x) \big| \leqslant c_{x/t < f'(\tilde{p}_1)}^c(u_0,f) \, t^{-\mu} \; ,
		\end{equation*}
		with $\displaystyle c_{x/t < f'(\tilde{p}_2)}^c(u_0,f) := \frac{1}{2 \pi} \, \frac{1}{\mu} \, \left\| \tilde{u} \right\|_{L^{\infty}(p_1,p_2)} + \frac{1}{2 \pi} \left( 4 \left\| \tilde{u} \right\|_{L^{\infty}(p_1,p_2)} + \left\| \tilde{u}' \right\|_{L^1(p_1,p_2)} \right) m_{\tilde{p}_1}^{-1}$, where we set $m_{\tilde{p}_1} := f'(p_1) - f'(\tilde{p}_1) > 0$.\\ So we can finally write
		\begin{equation*}
			\forall \, (t,x) \in \mathfrak{C}\big(f'(\tilde{p}_1), f'(\tilde{p}_2)\big)^c \qquad \big| u(t,x) \big| \leqslant c_{\tilde{I}}^c(u_0,f) \, t^{-\mu} \; ,
		\end{equation*}
		where
		\begin{equation} \label{c_I^c}
			c_{\tilde{I}}^c(u_0,f) := c_{x/t > f'(\tilde{p}_2)}^c(u_0,f) + c_{x/t < f'(\tilde{p}_1)}^c(u_0,f) \; .
		\end{equation}
	\end{enumerate}
\end{proof}

\vspace{0.4cm}

An $L^{\infty}$-norm estimate can be easily derived from the preceding result.

\begin{SCHRO2} \label{SCHRO2}
	Suppose that $u_0$ satisfies Condition (C$_{[p_1,p_2],\mu}$) and choose two finite real numbers $\tilde{p}_1, \tilde{p}_2$ such that $[p_1,p_2] \subsetneq [\tilde{p}_1, \tilde{p}_2] =: \tilde{I} $. Then
		\begin{equation*}
			\forall \, t > 0 \qquad \big\| u(t,.) \big\|_{L^{\infty}(\R)} \leqslant c(u_0,f) \, t^{- \frac{\mu}{2}} + c_{\tilde{I}}^c(u_0,f) \, t^{- \mu} \; ,
		\end{equation*}
		where the constants $c(u_0,f) > 0$ and $c_{\tilde{I}}^c(u_0,f) > 0$ are given by \eqref{c_I} and \eqref{c_I^c} respectively. In particular, we have
		\begin{equation*}
			\forall \, t \geqslant 1 \qquad \big\| u(t,.) \big\|_{L^{\infty}(\R)} \leqslant \big( c(u_0,f) + c_{\tilde{I}}^c(u_0,f) \big) \, t^{- \frac{\mu}{2}} \; .
		\end{equation*}
\end{SCHRO2}

\begin{proof}
	Simple consequence of Theorem \ref{SCHRO1}.
\end{proof}

\vspace{0.4cm}

Thanks to a combination of the preceding result and of \cite{article1}, one can prove the optimality of the last $L^{\infty}$-norm estimate in the case of the free Schrödinger equation on the line.

\begin{OPTI}
	Let $u_S : \R_+^* \times \R \longrightarrow \C$ be the solution of
	\begin{equation*}
	\left\{ \begin{array}{rl}
			& \hspace{-2mm} \left[ i \, \partial_t - \partial_x^2 \right] u(t,x) = 0 \\ [2mm]
			& \hspace{-2mm} u(0,x) = u_0(x)
	\end{array} \right. \; ,
\end{equation*}
	for all $t > 0$ and $x \in \R$, where $u_0$ satisfies Condition (C$_{[p_1,p_2],\mu}$). Then
	\begin{equation} \label{optischro}
		\forall \, t \geqslant 1 \qquad \big\| u_S(t,.) \big\|_{L^{\infty}(\R)} \leqslant c(u_0) \, t^{- \frac{\mu}{2}} \; ,
	\end{equation}
	where the constant $c(u_0) > 0$ is given by Corollary \ref{SCHRO2}, and the decay rate is optimal.
\end{OPTI}

\begin{proof}
	Corollary \ref{SCHRO2} applied to the case $f(p) = p^2$ furnishes \eqref{optischro}. Then by supposing that $\tilde{u} \in \mathcal{C}^1\big( [p_1,p_2] \big)$ and that $\tilde{u}(p_2) = 0$, Theorem 2.6 of \cite{article1} is applicable and provides an estimate to one term of the solution on the space-time direction direction $x = 2 p_1 \, t$, namely,
	\begin{equation*}
		\forall \, t \geqslant 1 \qquad \left| u_S(t,x) - L_{\mu}(t,u_0) \, t^{-\frac{\mu}{2}} \right| \leqslant \tilde{c}(u_0) \, t^{-\frac{1}{2}} \; ;
	\end{equation*}
	the coefficient $L_{\mu}(t,u_0) \in \C$ is uniformly bounded with respect to time and the constant $\tilde{c}(u_0) > 0$ does not depend on $t$. This shows that the solution behaves like $L_{\mu}(t,u_0) \, t^{-\frac{\mu}{2}}$ on the above mentioned direction when $t$ tends to infinity, implying the optimality of the decay rate.
\end{proof}

\vspace{0.4cm}

In the following result, we furnish estimates of the solution in arbitrary narrow cones containing the direction $\frac{x}{t} = f'(p_1)$. In such regions, the phase has a stationary point which is in a neighbourhood of the singularity $p_1$. It is then expected that these two particular points interact with each other to produce the slow decay $t^{-\frac{\mu}{2}}$. The other frequencies are regular and so they do not contribute to this slow decay rate.\\
Here we can remove the frequency band condition: we consider an initial data having a Fourier transform supposed integrable on $\R$, combined with some technical hypotheses. Note that the following condition is inspired from Assumption (A$_{p_1,\mu}$). \\

\noindent \textbf{Condition ($\mathbf{C_{p_1, \mu}}$).} Fix $\mu \in (0,1]$ and choose a finite real number $p_1$.\\ A tempered distribution $u_0$ satisfies Condition (C$_{p_1, \mu}$) if and only if $\tf u_0 \in L^1(\R)$ and there exists a bounded differentiable function $\tilde{u} : \R \longrightarrow \C$ such that $\tilde{u}(p_1) \neq 0$ if $\mu \neq 1$, $\tilde{u}' \in L^1(\R)$ and
\begin{equation*}
	\forall \, p \in \R \backslash \{ p_1 \} \qquad \tf u_0(p) = |p-p_1|^{\mu-1} \tilde{u}(p) \; .
\end{equation*}

\begin{REM8}
	\emph{The integrability of $\tf u_0$ assures that an initial condition $u_0$ satisfying the above condition exists as a tempered distribution. Moreover the solution formula \eqref{formula} is still well-defined for $t > 0$ and $x \in \R$.}
\end{REM8}

\begin{PREGEO1} \label{PREGEO1}
	Suppose that $u_0$ satisfies Condition \emph{(C$_{p_1, \mu}$)} and choose two finite real numbers $\eta > \varepsilon > 0$. Then for all $\displaystyle (t,x) \in \mathfrak{C} \big( f'(p_1 - \varepsilon), f'(p_1 + \varepsilon) \big)$, we have
		\begin{equation*}
			\big| u(t,x) \big| \leqslant c_{\eta}^{(1)}(u_0,f) \, t^{-\frac{\mu}{2}} + c_{\eta, \varepsilon}^{(2)}(u_0,f) \, t^{-1} \; .
		\end{equation*}
		The constants $c_{\eta}^{(1)}(u_0,f)$ and $c_{\eta, \varepsilon}^{(2)}(u_0,f)$ are given by \eqref{c_1} and \eqref{c_2} respectively.
\end{PREGEO1}

\begin{proof}
	We shall employ the rewritting of the solution given in the proof of Theorem \ref{SCHRO1}, i.e.
	\begin{equation*}
		\forall \, (t,x) \in (0,+\infty) \times \R \qquad u(t,x) = \int_{\R} U(p) \, e^{i t \psi(p)} \, dp \; ,
	\end{equation*}
	where
	\begin{equation*}
		\left\{ \begin{array}{rl}
				& \displaystyle \forall \, p \in \R \backslash \{ p_1 \} \qquad U(p) := \frac{1}{2\pi}  \, \tf u_0(p) = \frac{1}{2\pi}  \, |p-p_1|^{\mu-1} \tilde{u}(p) \; , \\
				& \vspace{-0.3cm} \\
				& \displaystyle \forall \, p \in \R \qquad \psi(p) := \frac{x}{t} \, p - f(p) \; .
		\end{array} \right.
	\end{equation*}
	Let $\eta > \varepsilon > 0$ and split the integral as follows,
	\begin{align*}
		\int_{\R} U(p) \, e^{i t \psi(p)} \, dp	& = \int_{p_1- \eta}^{p_1 + \eta} \dots \; + \int_{\R \backslash [p_1 - \eta, p_1 + \eta]} \dots \\[3mm]
												& =: I^{(1)}(t,x,\eta) + I^{(2)}(t,x,\eta) \; .
	\end{align*}
	Firstly we study $I^{(1)}(t,x,\eta)$. We recall that
	\begin{equation*}
		\psi'(p) = \frac{x}{t} - f'(p) \; ;
	\end{equation*}
	since $\displaystyle \frac{x}{t}$ is supposed to belong to $\displaystyle \big[ f'(p_1 - \varepsilon), f'(p_1 + \varepsilon) \big]$, then $\psi$ has a stationary point which belongs to $[p_1-\varepsilon, p_1 + \varepsilon] \subset [p_1-\eta, p_1 + \eta]$. Following the arguments of the point i) of the proof of Theorem \ref{SCHRO1}, we apply Theorem \ref{VDC} on $[p_1 - \eta, p_1]$ and on $[p_1, p_1 + \eta]$ with $\rho = 2$, leading to
	\begin{equation*}
		\Big| I^{(1)}(t,x,\eta) \Big| \leqslant \left| \int_{p_1-\eta}^{p_1} \dots \right| + \left| \int_{p_1}^{p_1 + \eta} \dots \right| \leqslant \left( c_{1,\eta}^{(1)}(u_0,f) + c_{2,\eta}^{(1)}(u_0,f) \right) t^{-\frac{\mu}{2}} \; ,
	\end{equation*}
	where
	\begin{align*}
		& \bullet \quad c_{1,\eta}^{(1)}(u_0,f) := \frac{1}{2 \pi} \, \frac{3}{\mu} \, \left\| \tilde{u} \right\|_{L^{\infty}(p_1-\eta, p_1)} + \frac{1}{\pi} \left( 4 \left\| \tilde{u} \right\|_{L^{\infty}(p_1-\eta, p_1)} + \left\| \tilde{u}' \right\|_{L^1(p_1-\eta, p_1)} \right) m_{1,\eta}^{-1} \; , \\
		& \bullet \quad c_{2,\eta}^{(1)}(u_0,f) := \frac{1}{2 \pi} \, \frac{3}{\mu} \, \left\| \tilde{u} \right\|_{L^{\infty}(p_1, p_1 + \eta)} + \frac{1}{\pi} \left( 4 \left\| \tilde{u} \right\|_{L^{\infty}(p_1, p_1 + \eta)} + \left\| \tilde{u}' \right\|_{L^1(p_1, p_1 + \eta)} \right) m_{2,\eta}^{-1} \; ,
	\end{align*}
	with $\displaystyle m_{1, \eta} := \min_{p \in [p_1-\eta, p_1]} f''(p) > 0$ and $\displaystyle m_{2, \eta} := \min_{p \in [p_1, p_1 + \eta]} f''(p) > 0$. So we define the constant $c_{\eta}^{(1)}(u_0, f)$ by
	\begin{equation} \label{c_1}
		c_{\eta}^{(1)}(u_0, f) := c_{1,\eta}^{(1)}(u_0,f) + c_{2,\eta}^{(1)}(u_0,f) \; .
	\end{equation}
	Let us study $I^{(2)}(t,x,\eta)$. Let $k \in \N$ and consider the following sequence,
	\begin{equation*}
		\tilde{I}_k^{(2)}(t,x, \eta) := \int_{p_1+\eta}^{p_1+ \eta + k} U(p) \, e^{it \psi(p)} \, dp \; .
	\end{equation*}
	Since $\displaystyle \frac{x}{t} \in \big[ f'(p_1 - \varepsilon) , f'(p_1 + \varepsilon)  \big]$, we note that the first derivative of the phase function does not vanish on $[p_1 + \eta, p_1 + \eta + k]$ and more precisely, we have for any $k \in \N$,
	\begin{equation*}
		\forall \, p \in [p_1 + \eta, p_1 + \eta + k] \quad \big| \psi'(p) \big| = f'(p) - \frac{x}{t} \geqslant f'(p_1 + \eta) - f'(p_1 + \varepsilon) =: \tilde{m}_{1, \eta, \varepsilon} > 0  \; .
	\end{equation*}
	Theorem \ref{VDC4} in the case $\mu = 1$ furnishes for all $(t,x) \in \mathfrak{C}\big( f'(p_1 - \varepsilon), f'(p_1 + \varepsilon) \big)$,
	\begin{equation*}
		\left| \tilde{I}_k^{(2)}(t,x, \eta) \right| \leqslant \frac{1}{2\pi} \left( 3 \, \left\| U \right\|_{L^{\infty}(p_1+\eta, p_1 + \eta + k)} +  \left\| U' \right\|_{L^1(p_1+\eta, p_1 + \eta + k)} \right) \tilde{m}_{1, \eta, \varepsilon}^{-1} \, t^{-1} \; .
	\end{equation*}
	But we have
	\begin{equation*}
		\forall \, p \in [p_1 + \eta, p_1 + \eta + k] \qquad \left| U(p) \right| \leqslant \eta^{\mu-1} \left\| \tilde{u} \right\|_{L^{\infty}(\R)} \; ,
	\end{equation*}
	and
	\begin{equation*}
		\int_{p_1 + \eta}^{p_1 + \eta + k} \left| U'(p) \right| dp \leqslant \eta^{\mu-1} \left( \| \tilde{u} \|_{L^{\infty}(\R)} + \| \tilde{u}' \|_{L^1(\R)} \right) \; .
	\end{equation*}
	Consequently, $\tilde{I}_2^{(k)}(t,x,\eta)$ can be estimated as follows,
	\begin{equation} \label{tildeI2}
		\left| \tilde{I}_k^{(2)}(t,x, \eta) \right| \leqslant \frac{1}{2\pi} \, \eta^{\mu-1} \left( 4 \, \left\| \tilde{u} \right\|_{L^{\infty}(\R)} +  \left\| \tilde{u}' \right\|_{L^1(\R)} \right) \tilde{m}_{1, \eta, \varepsilon}^{-1} \, t^{-1} \; .			
	\end{equation}
	Using the dominated convergence Theorem which claims that
	\begin{equation*}
		\lim_{k \rightarrow + \infty} \tilde{I}_k^{(2)}(t,x,\eta) = \int_{p_1 + \eta}^{+\infty} U(p) \, e^{i t \psi(p)} dp \; ,
	\end{equation*}
	we can take the limit in \eqref{tildeI2} providing
	\begin{equation*}
		\left| \int_{p_1 + \eta}^{+ \infty} U(p) \, e^{i t \psi(p)} dp \right| \leqslant c_{1,\eta,\varepsilon}^{(2)}(u_0, f) \, t^{-1} \; ,
	\end{equation*}
	with
	\begin{equation*}
		c_{1,\eta,\varepsilon}^{(2)}(u_0, f) := \frac{1}{2\pi} \, \eta^{\mu-1} \left( 4 \, \left\| \tilde{u} \right\|_{L^{\infty}(\R)} +  \left\| \tilde{u}' \right\|_{L^1(\R)} \right) \tilde{m}_{1, \eta, \varepsilon}^{-1} \; .
	\end{equation*}
	Similar arguments permit to furnish
	\begin{equation*}
		\forall \, (t,x) \in \mathfrak{C}\big( f'(p_1-\varepsilon), f'(p_1 + \varepsilon) \big) \qquad \left| \int_{- \infty}^{p_1 - \eta} U(p) \, e^{i t \psi(p)} dp \right| \leqslant c_{2,\eta,\varepsilon}^{(2)}(u_0, f) \, t^{-1} \; ,
	\end{equation*}
	with
	\begin{equation*}
		c_{2,\eta,\varepsilon}^{(2)}(u_0, f) := \frac{1}{2\pi} \, \eta^{\mu-1} \left( 4 \, \left\| \tilde{u} \right\|_{L^{\infty}(\R)} +  \left\| \tilde{u}' \right\|_{L^1(\R)} \right) \tilde{m}_{2, \eta, \varepsilon}^{-1} \; ,
	\end{equation*}
	with $\displaystyle \tilde{m}_{2, \eta, \varepsilon}^{-1} := f'(p_1 - \varepsilon) - f'(p_1 - \eta) > 0$.\\
	Finally, by setting
	\begin{equation} \label{c_2}
		c_{\eta,\varepsilon}^{(2)}(u_0, f) := c_{1,\eta,\varepsilon}^{(2)}(u_0, f) + c_{2,\eta,\varepsilon}^{(2)}(u_0, f) \; ,
	\end{equation}
	we obtain for all $(t,x) \in \mathfrak{C}\big( f'(p_1-\varepsilon), f'(p_1 + \varepsilon) \big)$,
	\begin{equation*}
		\Big| I^{(2)}(t,x,\eta) \Big| \leqslant \left| \int_{-\infty}^{p_1-\eta} \dots \right| + \left| \int_{p_1 + \eta}^{+\infty} \dots \right| \leqslant c_{\eta,\varepsilon}^{(2)}(u_0, f) \, t^{-1} \; .
	\end{equation*}
	This ends the proof.
\end{proof}

\vspace{0.4cm}

On the other hand, in cones without the critical direction, the stationary point and the singularity are sufficiently far and so they do not interact with each other. Consequently, theirs effects on the decay rate are not coupled as in the previous result. Hence these two particular points provide two distinct decay rates: $t^{-\frac{1}{2}}$ coming from the stationary point and $t^{-\mu}$ coming from the singularity. We note that these two rates are better than $t^{-\frac{\mu}{2}}$.\\
This result combined with Theorem \ref{PREGEO1} highlights the fact that the singular frequency has a stronger influence on the decay rate in narrow regions containing the direction $p_1$, in accordance with the more precise results in \cite{article2}, which have been obtained under stronger conditions.

\begin{PREGEO2} \label{PREGEO2}
	Suppose that $u_0$ satisfies Condition \emph{(C$_{p_1, \mu}$)} and choose two finite real numbers $\tilde{p}_1 < \tilde{p}_2$ such that $p_1 \notin [\tilde{p}_1, \tilde{p}_2]$. Then for all $\displaystyle (t,x) \in \mathfrak{C}\big( f'(\tilde{p}_1), f'(\tilde{p}_2) \big)$, we have
		\begin{equation*}
			\big| u(t,x) \big| \leqslant c_{\tilde{p}_1, \tilde{p}_2}^{(1)}(u_0,f) \, t^{-\frac{1}{2}} + c_{\tilde{p}_1, \tilde{p}_2}^{(2)}(u_0,f) \, t^{-\mu} + c_{\tilde{p}_1, \tilde{p}_2}^{(3)}(u_0,f) \, t^{-1} \; .
		\end{equation*}
		The constants $c_{\tilde{p}_1, \tilde{p}_2}^{(1)}(u_0,f)$, $c_{\tilde{p}_1, \tilde{p}_2}^{(2)}(u_0,f)$ and $c_{\tilde{p}_1, \tilde{p}_2}^{(3)}(u_0,f)$ are given by \eqref{c_21}, \eqref{c_22} and \eqref{c_23} respectively.
\end{PREGEO2}

\begin{proof}
	The employed arguments in this proof are similar to those of the preceding one, so we only furnish a sketch of the proof. \\
	Let $\eta \in \big(0, \min \{ |\tilde{p}_1 - p_1|, |p_1 - \tilde{p}_2| \} \big)$ and split the integral again,
	\begin{align*}
		u(t,x) = \int_{\R} U(p) \, e^{it \psi(p)} \, dp	& = \int_{\tilde{p}_1 - \eta}^{\tilde{p}_2 + \eta} \dots \; + \int_{\R \backslash [\tilde{p}_1 - \eta, \tilde{p}_2 + \eta]} \dots \\[3mm]
												& =: I^{(1)}(t,x,\eta) + I^{(2)}(t,x, \eta) \; ,
	\end{align*}
	On the interval $[\tilde{p}_1- \eta, \tilde{p}_2 + \eta]$, the phase has a unique stationary point and the amplitude is regular. Theorem \ref{VDC} is applicable with $\rho = 2$ and $\mu = 1$, and we get
	\begin{equation*}
		\forall \, (t,x) \in \mathfrak{C}\big( f'(\tilde{p}_1), f'(\tilde{p}_2) \big) \qquad \Big| I^{(1)}(t,x,\eta) \Big| \leqslant c_{\tilde{p}_1, \tilde{p}_2}^{(1)}(u_0,f) \, t^{-\frac{1}{2}} \; ,
	\end{equation*}
	where
	\begin{equation} \label{c_21}
		c_{\tilde{p}_1, \tilde{p}_2}^{(1)}(u_0,f) := \left\{ \begin{array}{rl}
				& \displaystyle \frac{(\tilde{p}_1 - \eta - p_1)^{\mu-1}}{\pi} \, \bigg( \left\| \tilde{u} \right\|_{L^{\infty}(\R)} \\
				& \displaystyle \qquad + \; \Big( 4  \left\| \tilde{u} \right\|_{L^{\infty}(\R)} + \left\| \tilde{u}' \right\|_{L^1(\R)} \Big) m_{1, \tilde{p}_1, \tilde{p}_2}^{-1} \bigg) \; , \quad \text{if} \; p_1 < \tilde{p}_1 \; , \\
				& \vspace{-0.3cm} \\
				& \displaystyle \frac{(p_1 - \tilde{p}_2 - \eta)^{\mu-1}}{\pi} \, \bigg( \left\| \tilde{u} \right\|_{L^{\infty}(\R)} \\
				& \displaystyle \qquad + \; \Big( 4  \left\| \tilde{u} \right\|_{L^{\infty}(\R)} + \left\| \tilde{u}' \right\|_{L^1(\R)} \Big) m_{1, \tilde{p}_1, \tilde{p}_2}^{-1} \bigg) \; , \quad \text{if} \; p_1 > \tilde{p}_2 \; ,
		\end{array} \right.
	\end{equation}
	with $\displaystyle m_{1, \tilde{p}_1, \tilde{p}_2} := \min_{p \in [\tilde{p}_1 - \eta, \tilde{p}_2 + \eta]} f''(p) > 0$.\\
	Now let us study $I^{(2)}(t,x,\eta)$. First of all, we remark that we integrate over two infinite branches such that one of them contains the singularity $p_1$. Consequently we shall suppose that $p_1 < \tilde{p}_1$ without loss of generality; the other case $p_1 > \tilde{p}_2$ can be treated in a similar way. We consider the following sequence,
	\begin{equation*}
		\forall \, k \in \N^* \qquad \tilde{I}_k^{(2)}(t,x,\eta) := \int_{p_1-k}^{\tilde{p}_1-\eta} U(p) \, e^{it \psi(p)} \, dp \; .
	\end{equation*}
	We note that $[p_1 - k, \tilde{p}_1 - \eta]$ contains the singularity $p_1$ and $\psi'$ does not vanish on this interval. Hence Theorem \ref{VDC4} can be employed on $[p_1 - k, p_1]$ and on $[p_1, \tilde{p}_1 - \eta]$, and taking the limit with the dominated convergence Theorem, we obtain
	\begin{equation*}
		\forall \, (t,x) \in \mathfrak{C}\big( f'(\tilde{p}_1), f'(\tilde{p}_2) \big) \qquad \left| \int_{-\infty}^{\tilde{p}_1-\eta} U(p) \, e^{it \psi(p)} \, dp \right| \leqslant c_{\tilde{p}_1, \tilde{p}_2}^{(2)}(u_0,f) \, t^{-\mu} \; ,
	\end{equation*}
	where
	\begin{equation*}
		c_{\tilde{p}_1, \tilde{p}_2}^{(2)}(u_0,f) := \frac{1}{\pi} \left( \frac{1}{\mu} \left\| \tilde{u} \right\|_{L^{\infty}(\R)} + \Big( 4 \left\| \tilde{u} \right\|_{L^{\infty}(\R)} + \left\| \tilde{u}' \right\|_{L^1(\R)} \Big) \tilde{m}_{2, \tilde{p}_1}^{-1} \right) \; ,
	\end{equation*}	
	with $\displaystyle \tilde{m}_{2, \tilde{p}_1} := f'(\tilde{p}_1) - f'(\tilde{p}_1 - \eta) > 0$. On the other infinite branch, we define $\tilde{I}_k^{(3)}(t,x,\eta)$ as follows,
	\begin{equation*}
		\forall \, k \in \N^* \qquad \tilde{I}_k^{(3)}(t,x,\eta) := \int_{\tilde{p}_2+\eta}^{\tilde{p}_2+\eta+k} U(p) \, e^{it \psi(p)} \, dp \; .
	\end{equation*}
	Here there is no singularity or stationary point, therefore Theorem \ref{VDC4} in the case $\mu = 1$ is applicable and furnishes
	\begin{equation*}
		\forall \, (t,x) \in \mathfrak{C}\big( f'(\tilde{p}_1), f'(\tilde{p}_2) \big) \qquad \left| \int_{\tilde{p}_2+\eta}^{+\infty} U(p) \, e^{it \psi(p)} \, dp \right| \leqslant c_{\tilde{p}_1, \tilde{p}_2}^{(3)}(u_0,f) \, t^{-1} \; ,
	\end{equation*}
	where
	\begin{equation*}
		c_{\tilde{p}_1, \tilde{p}_2}^{(3)}(u_0,f) := \frac{(\tilde{p}_2 + \eta - p_1)^{\mu-1}}{2 \pi} \left( 4 \left\| \tilde{u} \right\|_{L^{\infty}(\R)} + \left\| \tilde{u}' \right\|_{L^1(\R)} \right) \tilde{m}_{3, \tilde{p}_2}^{-1} \; ,
	\end{equation*}
	with $\displaystyle \tilde{m}_{3, \tilde{p}_2} := f'(\tilde{p}_2 + \eta) - f'(\tilde{p}_2) > 0$.\\
	Consequently we have an estimate of $I^{(2)}(t,x,\eta)$ in the case $p_1 < \tilde{p}_1$,
	\begin{equation*}
		\forall \, (t,x) \in \mathfrak{C}\big( f'(\tilde{p}_1), f'(\tilde{p}_2) \big) \qquad \Big| I^{(2)}(t,x,\eta) \Big| \leqslant c_{\tilde{p}_1, \tilde{p}_2}^{(2)}(u_0,f) \, t^{-\mu} + c_{\tilde{p}_1, \tilde{p}_2}^{(3)}(u_0,f) \, t^{-1} \; .
	\end{equation*}
	To conclude, we provide the values of the constants $c_{\tilde{p}_1, \tilde{p}_2}^{(2)}(u_0,f)$ and $c_{\tilde{p}_1, \tilde{p}_2}^{(3)}(u_0,f)$ depending on the position of $p_1$:
	\begin{equation} \label{c_22}
		\bullet \quad c_{\tilde{p}_1, \tilde{p}_2}^{(2)}(u_0,f) := \left\{ \begin{array}{rl}
				& \displaystyle \frac{1}{\pi} \left( \frac{1}{\mu} \left\| \tilde{u} \right\|_{L^{\infty}(\R)} + \Big( 4 \left\| \tilde{u} \right\|_{L^{\infty}(\R)} + \left\| \tilde{u}' \right\|_{L^1(\R)} \Big) \tilde{m}_{2, \tilde{p}_1}^{-1} \right) \; , \quad \text{if} \; p_1 < \tilde{p}_1 \; , \\
				& \vspace{-0.3cm} \\
				& \displaystyle \frac{1}{\pi} \left( \frac{1}{\mu} \left\| \tilde{u} \right\|_{L^{\infty}(\R)} + \Big( 4 \left\| \tilde{u} \right\|_{L^{\infty}(\R)} + \left\| \tilde{u}' \right\|_{L^1(\R)} \Big) \tilde{m}_{3, \tilde{p}_2}^{-1} \right) \; , \quad \text{if} \; p_1 > \tilde{p}_2 \; ,
		\end{array} \right.
	\end{equation}
	\begin{equation} \label{c_23}
		\bullet \qquad c_{\tilde{p}_1, \tilde{p}_2}^{(3)}(u_0,f) := \left\{ \begin{array}{rl}
				& \displaystyle \frac{(\tilde{p}_2 + \eta - p_1)^{\mu-1}}{2 \pi} \left( 4 \left\| \tilde{u} \right\|_{L^{\infty}(\R)} + \left\| \tilde{u}' \right\|_{L^1(\R)} \right) \tilde{m}_{3, \tilde{p}_2}^{-1} \; , \quad \text{if} \; p_1 < \tilde{p}_1 \; , \\
				& \vspace{-0.3cm} \\
				& \displaystyle \frac{(p_1 - \tilde{p}_1 + \eta)^{\mu-1}}{2 \pi} \left( 4 \left\| \tilde{u} \right\|_{L^{\infty}(\R)} + \left\| \tilde{u}' \right\|_{L^1(\R)} \right) \tilde{m}_{2, \tilde{p}_1}^{-1} \; , \quad \text{if} \; p_1 > \tilde{p}_2 \; .
		\end{array} \right.
	\end{equation}
\end{proof}

\vspace{0.4cm}

It could be interesting to estimate the $L^{\infty}$-norm of the solution with an initial data satisfying Condition (C$_{p_1, \mu}$). Unfortunately, the estimates of the two previous theorems are not necessarily uniform. For example, the number $m_{1,\tilde{p}_1, \tilde{p}_2}$ defined above may blow up when $\tilde{p}_1$ or $\tilde{p}_2$ tends to infinity if $\displaystyle \lim_{p \rightarrow \infty} f''(p) = 0$. Therefore, in the general case, it does not seem possible to derive a global estimate in a direct way. To establish such an estimate, we employ another method which consists in splitting the frequency line in small frequency bands and studying the contribution of each band. To do so, we employ the uniform estimate given by Corollary \ref{SCHRO2}. Then we add up all these estimates to obtain the final result.\\
Nevertheless, we need some extra growth conditions on the phase and decay assumptions on the amplitude: under our hypotheses, the second derivative of the symbol is allowed to vanish at infinity; so the order of the stationary point of the phase may change at infinity. To prevent a possible influence coming from this change of the nature of the phase, we choose an initial data having a Fourier transform which tends sufficiently fast to $0$ at infinity. Hence we have to consider new assumptions on the initial data which are a little more restrictive as compared with Condition (C$_{p_1, \mu}$). \\

\noindent \textbf{Condition ($\mathbf{C_{\mu, \alpha,r}}$).} Fix $\mu \in (0,1]$, $\alpha \geqslant 0$ and $r \geqslant 0$.\\
A tempered distribution $u_0$ satisfies Condition (C$_{\mu, \alpha,r}$) if and only if there exists a bounded differentiable function $\tilde{u} : \R \longrightarrow \C$ such that $\tilde{u}(0) \neq 0$ if $\mu \neq 1$, with
\begin{equation*}
	\forall \, p \in \R \backslash \{ 0 \} \qquad \tf u_0(p) = |p|^{\mu-1} \tilde{u}(p) \; .
\end{equation*}
Moreover we suppose that
\begin{equation*}
	\exists \,  M \geqslant 0 \quad \forall \, p \in \R \qquad \big| \tilde{u}(p) \big| \leqslant M \left( 1 + p^2 \right)^{-\frac{\alpha}{2}} \; ,
\end{equation*}
and that $\displaystyle \tilde{u}' \in L_{loc}^1(\R)$ with
\begin{equation*}
	\exists \, M' \geqslant 0 \quad \forall \, n \in \left\{ n \in \mathbb{Z} \, \big| \, |n| \geqslant r \right\} \qquad \left\| \tilde{u}' \right\|_{L^1(n,n+1)} \leqslant M' \, |n|^{-\alpha} \; .
\end{equation*}	

\begin{REM4}
	\em
	\begin{enumerate}
		\item Here we put the singular frequency at $0$ only for simplicity.
		\item If we suppose $\alpha > \mu$ then $\tf u_0 \in L^1(\R)$. Indeed $\displaystyle \tf u_0 \in L_{loc}^1(\R)$ since $\mu \in (0,1]$ and $\displaystyle \tilde{u} \in L^{\infty}(\R)$. Furthermore we have
	\begin{equation*}
		\forall \, p \in \R \backslash \{0 \} \qquad \big| \tf u_0(p) \big| \leqslant M \left( 1 + p^2 \right)^{-\frac{\alpha}{2}} |p|^{\mu-1} \leqslant M \, |p|^{\mu-1-\alpha} \; ,
	\end{equation*}
	since $\big( 1 + p^2 \big)^{\frac{1}{2}} \geqslant |p|$. Hence the hypothesis $\alpha > \mu$ implies the integrability of $\tf u_0$ on $\R$.\\
	The fact that $\tf u_0 \in L^1(\R)$ implies the existence of $u_0$ in $\mathcal{S}'(\R)$ and the solution formula \eqref{formula} is still well-defined for all $t > 0$ and $x \in \R$. 
	\item Let us give an example of the above condition. Choose $u_0 \in \mathcal{S}'(\R)$ such that its Fourier transform has the following form:
	\begin{equation*}
		\forall \, p \in \R \backslash \{ 0 \} \qquad \tf u_0 (p) = |p|^{\mu-1} \, (1+p^2)^{-\frac{\alpha}{2}} \; ,
	\end{equation*}
	with $\mu \in (0,1]$ and $\alpha > \mu$. Here $\tilde{u} : \R \longrightarrow \R$ is defined by $\displaystyle \tilde{u}(p) = (1+p^2)^{-\frac{\alpha}{2}}$ for all $p \in \R$. \\
	Here we only have to control $\left\| \tilde{u}' \right\|_{L^1(n,n+1)}$ since the other hypotheses are clearly satisfied. One can quickly show that
	\begin{equation*}
		\left\| \tilde{u}' \right\|_{L^1(n,n+1)} = \left\{ \begin{array}{rl}
				& \displaystyle \tilde{u}(n) - \tilde{u}(n+1) \leqslant \tilde{u}(n) \; , \quad \text{if} \; n \geqslant 0 \; , \\
				& \vspace{-0.3cm} \\
				& \displaystyle \tilde{u}(n+1) - \tilde{u}(n) \leqslant \tilde{u}(n+1) \; , \quad \text{if} \; n \leqslant -1 \; .
		\end{array} \right.
	\end{equation*}
	Using the fact that $|n+1|^{-\alpha} \leqslant 2^{\alpha} |n|^{-\alpha}$, if $n \leqslant -2$ according to Lemma \ref{LEM2}, we obtain
	\begin{equation*}
		\left\| \tilde{u}' \right\|_{L^1(n,n+1)} \leqslant \left\{ \begin{array}{rl}
				& \displaystyle \left( 1 + n^2 \right)^{-\frac{\alpha}{2}} \leqslant n^{-\alpha} \; , \quad \text{if} \; n \geqslant 0 \; , \\
				& \vspace{-0.3cm} \\
				& \displaystyle \left( 1 + (n+1)^2 \right)^{-\frac{\alpha}{2}} \leqslant |n+1|^{-\alpha} \leqslant 2^{\alpha} |n|^{-\alpha} \; , \quad \text{if} \; n \leqslant -2 \; .
		\end{array} \right.
	\end{equation*}
	Hence for all $|n| \geqslant 2$, we have
	\begin{equation*}
		\left\| \tilde{u}' \right\|_{L^1(n,n+1)} \leqslant 2^{\alpha} |n|^{-\alpha} \; ,
	\end{equation*}
	and consequently, $u_0$ satisfies Condition (C$_{\mu, \alpha,2}$).
	\end{enumerate}
\end{REM4}

\begin{SCHRO3} \label{SCHRO3}
	We suppose that
	\begin{equation*}
		\exists \, R > 0 \quad \forall \, |p| \geqslant R \qquad c \, |p|^{-\beta} \leqslant f''(p) \; ,
	\end{equation*}
	for certain $\beta, c > 0$. Moreover we assume that $u_0$ satisfies Condition \emph{(C$_{\mu,\alpha,r}$)}, with $\mu \in (0,1]$, $\alpha > \mu + \beta$ and $r \leqslant R$. Then
	\begin{equation*}
		\forall \, t \geqslant 1 \qquad \big\| u(t,.) \big\|_{L^{\infty}(\R)} \leqslant c^{(1)}(u_0,f) \, t^{- \frac{\mu}{2}} + c^{(2)}(u_0,f) \, t^{- \frac{1}{2}} \; ,
	\end{equation*}
	where the constants $c^{(1)}(u_0,f)$ and $c^{(2)}(u_0,f)$ are given by \eqref{C_1} and \eqref{C_(2)}, respectively.
\end{SCHRO3}


\begin{proof}
	We recall that the solution of the initial value problem can be written as follows,
	\begin{equation*}
		\forall \, (t,x) \in (0,+\infty) \times \R \qquad u(t,x) = \int_{\R} U(p) \, e^{i t \psi(p)} \, dp \; ,
	\end{equation*}
	where
	\begin{equation*}
		\left\{ \begin{array}{rl}
				& \displaystyle \forall \, p \in \R \backslash \{0 \} \qquad U(p) := \frac{1}{2\pi}  \, \tf u_0(p) = \frac{1}{2\pi}  \, |p|^{\mu-1} \tilde{u}(p) \; , \\
				& \vspace{-0.3cm} \\
				& \displaystyle \forall \, p \in \R \qquad \psi(p) := \frac{x}{t} \, p - f(p) \; .
		\end{array} \right.
	\end{equation*}
	Let us define
		\begin{equation*}
			N := \lceil R \rceil + 1 \qquad , \qquad \mathfrak{S}_N = \mathbb{Z} \backslash \{-N,\dots, N-1 \} \; ,
		\end{equation*}
		where $\lceil . \rceil$ is the ceiling function. Now we split the integral,
		\begin{align}
			\int_{\R} U(p) e^{it \psi(p)} \, dp	& = \int_{\R} \chi_{[-N,N)}(p) U(p) \, e^{it \psi(p)} \, dp + \int_{\R} \sum_{n \in \mathfrak{S}_N} \chi_{[n,n+1)}(p) U(p) \, e^{it \psi(p)} \, dp \nonumber \\
												& \label{split} = \int_{-N}^{N} U(p) \, e^{it \psi(p)} \, dp + \sum_{n \in \mathfrak{S}_N} \int_n^{n+1} U(p) \, e^{it \psi(p)} \, dp \; ,
		\end{align}
		where $\chi_{[n,n+1)}$ is the characteristic function of the interval $[n,n+1)$. We can apply Corollary \ref{SCHRO2} on $[-N,0]$ (resp. on $[0,N]$), with $\tilde{I}=[-N-1,1]$ (resp. $\tilde{I} = [-1,N+1]$), and we obtain
		\begin{align}
			\forall \, (t,x) \in \{t \geqslant 1 \} \times \R \qquad \left| \int_{-N}^N U(p) \, e^{i t \psi(p)} \, dp \right|	& \leqslant \left| \int_{-N}^{0} \dots \right| + \left| \int_0^N \dots \right| \nonumber \\
						& \leqslant \big(c_{-N}(u_0,f) + c_{+N}(u_0,f) \big) \, t^{-\frac{\mu}{2}} \nonumber \\
						& \label{C_1} =: c^{(1)}(u_0,f) \, t^{-\frac{\mu}{2}} \; ,
		\end{align}
		with
		\begin{align*}
			& \bullet \quad c_{-N} (u_0,f) := \frac{1}{2 \pi} \, \frac{5}{\mu} \left\| \tilde{u} \right\|_{L^{\infty}(-N,0)} + \frac{1}{2 \pi} \, \Big( 4 \left\| \tilde{u} \right\|_{L^{\infty}(-N,0)} + \left\| \tilde{u}' \right\|_{L^1(-N,0)} \Big) \, m_{-N} \; , \\
			& \bullet \quad c_{+N} (u_0,f) := \frac{1}{2 \pi} \, \frac{5}{\mu} \left\| \tilde{u} \right\|_{L^{\infty}(0,N)} + \frac{1}{2 \pi} \, \Big( 4 \left\| \tilde{u} \right\|_{L^{\infty}(0,N)} + \left\| \tilde{u}' \right\|_{L^1(0,N)} \Big) \, m_{+N} \; ,
		\end{align*}
		and
		\begin{equation*}
			m_{\pm N} := 2 \left( \min_{p \in [0, \pm N]} f''(p) \right)^{-1} + \Big( \pm f' \big( \pm (N+1) \big) \mp f'(\pm N) \Big)^{-1} + \Big(\pm f'(0) \mp f'(\mp1)\Big)^{-1} > 0 \; .
		\end{equation*}
		Now let us study each term of the series. By hypothesis, $U$ is regular on $[n,n+1]$ for $n \in \mathfrak{S}_N$. Corollary \ref{SCHRO2} is then applicable with $\mu = 1$, $\tilde{I}=[n-1,n+2]$, and it furnishes
		\begin{equation*}
			\left| \int_{n}^{n+1} U(p) \, e^{it \psi(p)} \, dp \right| \leqslant c_n(u_0,f) \, t^{-\frac{1}{2}} \; ;
		\end{equation*}
		the constant $c_n(u_0,f) > 0$ is given by
		\begin{equation*}
			c_n(u_0,f)	:= \frac{5}{2 \pi} \, \left\| U \right\|_{L^{\infty}(n,n+1)} + \frac{1}{2 \pi} \, \Big( 4 \left\| U \right\|_{L^{\infty}(n,n+1)} + \left\| U' \right\|_{L^1(n,n+1)} \Big) \, m_{n} \; ,
		\end{equation*}
		with $\displaystyle m_{n} := 2 \left( \min_{p \in [n,n+1]} f''(p) \right)^{-1} + \big( f'(n+2)-f'(n+1) \big)^{-1} + \big( f'(n) - f'(n-1) \big)^{-1} > 0$.\\ On the one hand, we have by using the hypothesis on $u_0$ and Lemma \ref{LEM2},
		\begin{equation*}
			\left\| U \right\|_{L^{\infty}(n,n+1)} \leqslant 2^{1-\mu} |n|^{\mu-1} \, M 2^{\alpha} |n|^{-\alpha} = 2^{1-\mu + \alpha} M \, |n|^{\mu-1-\alpha} \; .
		\end{equation*}
		Moreover
		\begin{align}
			\left\| U' \right\|_{L^1(n,n+1)}	& \leqslant \int_n^{n+1} (1-\mu) |p|^{\mu-2} \left| \tilde{u} (p) \right| dp + \int_n^{n+1} |p|^{\mu-1} \left| \tilde{u}' (p) \right| dp \nonumber \\
												& \leqslant \left\| \tilde{u} \right\|_{L^{\infty}(n,n+1)} \int_n^{n+1} (1-\mu) |p|^{\mu-2} dp + 2^{1-\mu} \, |n|^{\mu-1} \left\| \tilde{u}' \right\|_{L^1(n,n+1)} \nonumber \\
												& \label{estderivu} \leqslant M 2^{\alpha} |n|^{-\alpha} \, 2^{1-\mu} |n|^{\mu-1} + 2^{1-\mu} |n|^{\mu-1} \, M' |n|^{-\alpha} \\
												& \leqslant 2^{1-\mu} \left( 2^{\alpha} M + M' \right) |n|^{\mu-1-\alpha} \nonumber \; ,
		\end{align}
		where the additional hypothesis on $\tilde{u}'$ was used to get \eqref{estderivu}. On the other hand,
		\begin{equation*}
			\forall \, x , y \in \R \qquad f'(x) - f'(y) = \int_y^x f''(p) \, dp \geqslant |x-y| \min_{p \in [x,y]} f''(p) \; ,
		\end{equation*}
		and by the hypothesis on $f$, we have
		\begin{equation*}
			f''(p) \geqslant c \, |p|^{-\beta} \geqslant \left\{	\begin{array}{rl}
				& \displaystyle c \, 2^{-\beta} |n|^{-\beta} \; , \quad \text{if} \; p \in [n,n+1] \; , \\ [2mm]
				& \displaystyle c \, 2^{-\beta} |n+1|^{-\beta} \geqslant c \, 4^{-\beta} |n|^{-\beta} \; , \quad \text{if} \; p \in [n+1,n+2] \; , \\ [2mm]
				& \displaystyle c \, 2^{-\beta} |n-1|^{-\beta} \geqslant c \, 4^{-\beta} |n|^{-\beta} \; , \quad \text{if} \; p \in [n-1,n] \; ;
		\end{array} \right.
		\end{equation*}
		the last inequalities were obtained by employing Lemma \ref{LEM2}. This provides
		\begin{equation*}
			f'(n+2) - f'(n+1) \geqslant 4^{-\beta} c \, |n|^{-\beta} \qquad , \qquad f'(n) - f'(n-1) \geqslant 4^{-\beta} c \, |n|^{-\beta} \; .
		\end{equation*}
		It follows
		\begin{equation*}
			m_n \leqslant 2^{\beta} c^{-1} \, |n|^{\beta} + 4^{\beta} c^{-1} \, |n|^{\beta} + 4^{\beta} c^{-1} \, |n|^{\beta} \leqslant 3 \times 2^{2\beta} c^{-1} \, |n|^{\beta} \; .
		\end{equation*}
		Then we obtain
		\begin{align*}
			c_n(u_0,f)	& = \frac{5}{2 \pi} \, \left\| U \right\|_{L^{\infty}(n,n+1)} + \frac{1}{2 \pi} \, \Big( 4 \left\| U \right\|_{L^{\infty}(n,n+1)} + \left\| U' \right\|_{L^1(n,n+1)} \Big) \, m_{n} \\
						& \leqslant 5 \, \frac{2^{\alpha-\mu} M}{\pi} \, |n|^{\mu-1-\alpha} + 3 \, \frac{2^{2-\mu+\alpha+2\beta} M}{\pi \, c} \, |n|^{\mu-1-\alpha+\beta} \\
						& \qquad \qquad + 3 \, \frac{2^{-\mu+2\beta}  \left( 2^{\alpha} M + M' \right)}{\pi \, c} \, |n|^{\mu-1-\alpha+\beta} \; .
		\end{align*}
		Since $\alpha > \mu + \beta$, the sequence $\big\{ c_n(u_0,f) \big\}_{n \in \mathfrak{S}_N}$ is sommable. It follows
		\begin{equation*}
			\left| \sum_{n \in \mathfrak{S}_N} \int_n^{n+1} U(p) \, e^{it \psi(p)} \, dp \right|	\leqslant \sum_{n \in \mathfrak{S}_N} \left| \int_n^{n+1} U(p) \, e^{it \psi(p)} \, dp \right| \leqslant \left( \sum_{n \in \mathfrak{S}_N} c_n(u_0,f) \right) t^{-\frac{1}{2}} \; .
		\end{equation*}
		Then we can control the above series by employing the following estimate of the Riemann Zeta function,
		\begin{equation*}
			\forall \, \sigma > 1 \qquad \sum_{n \in \N^*} n^{-\sigma} \leqslant \frac{\sigma}{\sigma - 1} \; .
		\end{equation*}
		Hence
		\begin{align}
			\sum_{n \in \mathfrak{S}_N} c_n(u_0,f)	& \leqslant 5 \, \frac{2^{\alpha-\mu + 1} M}{\pi} \, \frac{\alpha + 1 - \mu}{\alpha - \mu} \, + 3 \, \frac{2^{3-\mu+\alpha+2\beta} M}{\pi \, c} \, \frac{\alpha + 1 - \mu - \beta}{\alpha - \mu - \beta} \nonumber \\
													& \qquad + 3 \, \frac{2^{1-\mu+2\beta}  \left( 2^{\alpha} M + M' \right)}{\pi \, c} \, \frac{\alpha + 1 - \mu - \beta}{\alpha - \mu - \beta} \nonumber \\
													& = 5 \, \frac{2^{\alpha-\mu + 1} M}{\pi} \, \frac{\alpha + 1 - \mu}{\alpha - \mu} \, + 3 \, \frac{2^{1-\mu+2\beta} ( 5 \times 2^{\alpha} M + M' )}{\pi \, c} \, \frac{\alpha + 1 - \mu - \beta}{\alpha - \mu - \beta} \nonumber \\
													& =: c^{(2)}(u_0,f) \; . \label{C_(2)}
		\end{align}
		Finally we obtain for all $t \geqslant 1$,
		\begin{equation*}
			\big\| u(t,.) \big\|_{L^{\infty}(\R)} \leqslant c^{(1)}(u_0,f) \, t^{-\frac{\mu}{2}} + c^{(2)}(u_0,f) \, t^{-\frac{1}{2}} \; .
		\end{equation*}
\end{proof}

\begin{REM7}
	\em Note that in the case $f'' \geqslant m > 0$ on $\R$, the $L^{\infty}$-norm estimate can be obtained easier: we consider the sequence of integrals $(I_n)_{n \geqslant 1}$ defined by
	\begin{equation*}
		\forall \, (t,x) \in \R_+^* \times \R \qquad I_n(t,x) = \int_{-n}^{n} \tf u_0(p) \, e^{-i t f(p) + ixp} dp \; ;
	\end{equation*}
	then we apply Corollary \ref{COR1} to $I_n(t,x)$, providing a uniform estimate which does not depend on $n$. Finally we take the limit employing the dominated convergence Theorem to conclude. In this case, we can relax the hypotheses on the initial data by supposing that $u_0$ satisfies Condition (C$_{p_1,\mu}$) with $p_1 = 0$.
\end{REM7}

\subsection*{An intrinsic concentration phenomenon caused by a limited growth of the symbol}

\hspace{0.5cm} In this last subsection, we exhibit the influence  of a growth limitation at infinity of the symbols on the dispersion of the solution.\\


We suppose that the second derivative of the symbol has a sufficient decay at infinity, implying the fact that the first derivative $f'$ is bounded on $\R$. An important consequence of this boundedness is the belonging of the stationary point to a bounded interval related to $f'(\R)$, leading to a space-time cone. Hence the influence of the stationary point on the decay is restricted to this cone.\\
Here we only provide an estimate outside the above mentioned cone: indeed, the preceding results, especially Theorem \ref{SCHRO3}, affirm that the $L^{\infty}$-norm of the solution is estimated by $t^{-\frac{\mu}{2}}$, covering the inside of the cone. In the following result, we refine this estimate outside the cone, furnishing the better decay $t^{-\mu}$.\\
To prove it, we employ the method of the proof of Theorem \ref{SCHRO3}: we assume that the initial data satisfy the above Condition (C$_{\mu, \alpha,r}$), then we decompose the frequency line and we study the influence of each frequency band.

\begin{SCHRO4} \label{SCHRO4}
	We suppose that
	\begin{equation*}
		\exists \, R > 0 \qquad \forall  \, |p| \geqslant R \qquad c_- \, |p|^{-\beta_-} \leqslant f''(p) \leqslant c_+ \, |p|^{-\beta_+} \; ,
	\end{equation*}
	for certain $\beta_- \geqslant \beta_+ > 1$ and $c_+ , c_- > 0$. Moreover we assume that $u_0$ satisfies Condition \emph{(C$_{\mu,\alpha,r}$)}, with $\mu \in (0,1]$, $\alpha > \mu + \beta_- - 1$ and $r \leqslant R$.	Then
	\begin{equation*}
		\forall \, (t,x) \in \mathfrak{C}(a,b)^c \qquad \big| u(t,x) \big| \leqslant c_1^c(u_0,f) \, t^{- \mu} + c_2^c(u_0,f) \, t^{-1} \; ,
	\end{equation*}
	where the constants $c_c^{(1)}(u_0,f)$ and $c_2^c(u_0,f)$ are given by \eqref{C_(1)} and \eqref{C_2}, respectively. The two finite real numbers $a < b$ verify
	\begin{equation*}
		\lim_{p \rightarrow - \infty} f'(p) = a \qquad , \qquad \lim_{p \rightarrow + \infty} f'(p) = b \; .
	\end{equation*}	
\end{SCHRO4}

\begin{REM6} \label{REM6}
	\em \begin{enumerate}
		\item Let us show that the above hypothesis on $f$ implies that $f'(\R) = (a,b)$, where the bounds are given by the limits of $f'$ at infinity.\\ Due to the fact that $f'' > 0$, $f'$ is strictly increasing and so if $f'$ is bounded, then it reaches its bounds at $\pm \infty$. On the compact interval $[-r,r]$, $f'$ is bounded since it is continuous. Now for $p \geqslant r$, we have
		\begin{equation*}
			f'(p) - f'(r) = \int_r^p f''(x) \, dx \leqslant c_+ \int_r^p x^{-\beta_+} dx \; ;
		\end{equation*}
		it follows
		\begin{equation*}
		f'(p) \leqslant \frac{c_+}{1-\beta_+} \, p^{1-\beta_+} + f'(r) - \frac{c_+}{1 - \beta_+} \, r^{1-\beta_+} \leqslant f'(r) - \frac{c_+}{1 - \beta_+} \, r^{1-\beta_+} < \infty \; .
		\end{equation*}
		Consequently $f'$ is bounded from above. Similar arguments show that $f'$ is bounded from below.
		\item Moreover we can control the distance between $f'$ and its bounds when $|p| \geqslant r$ :
		\begin{align*}
			& \bullet \quad \forall \, p \geqslant r \qquad b - f'(p) = \int_p^{+\infty} f''(x) \, dx \geqslant c_- \int_p^{+\infty} x^{-\beta_-} dx = - \frac{c_-}{1 - \beta_-} \, p^{1-\beta_-} \; , \\
			& \bullet \quad \forall \, p \leqslant -r \qquad f'(p) - a = \int_{-\infty}^p f''(x) \, dx \geqslant c_- \int_{-\infty}^p (-x)^{-\beta_-} dx = - \frac{c_-}{1 - \beta_-} \, (-p)^{1-\beta_-} \; .
		\end{align*}
	\end{enumerate}
\end{REM6}

\begin{proof}[Proof of Theorem \ref{SCHRO4}]
	We consider the solution formula given in the previous proofs and its rewriting once again. First let us split the integral as in \eqref{split}, that is to say,
	\begin{equation*}
		\int_{\R} U(p) e^{it \psi(p)} \, dp	 = \int_{-N}^{N} U(p) \, e^{it \psi(p)} \, dp + \sum_{n \in \mathfrak{S}_N} \int_n^{n+1} U(p) \, e^{it \psi(p)} \, dp \; ,
	\end{equation*}
	with
	\begin{equation*}
		N := \lceil R \rceil + 1 \qquad , \qquad \mathfrak{S}_N = \mathbb{Z} \backslash \{-N,\dots, N-1 \} \; .
	\end{equation*}
	We study the first integral. We note that the assumption $(t,x) \in \mathfrak{C}(a,b)^c$ implies that $\psi'$ does not vanish on $\R$. More precisely, on $[-N,N]$, we have
	\begin{align*}
		& \bullet \quad \forall \, p \in [-N,0] \quad \left| \psi'(p) \right| = \left| \frac{x}{t} - f'(p) \right| \geqslant \min \big\{ f'(-N) - a , b - f'(0) \big\} := m_{-N} > 0 \; , \\
		& \bullet \quad \forall \, p \in [0,N] \quad \left| \psi'(p) \right| = \left| \frac{x}{t} - f'(p) \right| \geqslant \min \big\{ f'(0) - a , b - f'(N) \big\} := m_{+N} > 0 \; ,
	\end{align*}
	and $\psi'$ is still monotone on $\R$ since $\psi'' = - f'' < 0$. Applying Theorem \ref{VDC4} on $[-N,0]$ and on $[0,N]$, we obtain
	\begin{align}
		\forall \, (t,x) \in \mathfrak{C}(a,b)^c \qquad \left| \int_{-N}^N U(p) \, e^{i t \psi(p)} \, dp \right|	& \leqslant \left| \int_{-N}^{0} \dots \right| + \left| \int_0^N \dots \right| \nonumber \\
					& \leqslant \big(c_{-N}^c(u_0,f) + c_{+N}^c(u_0,f) \big) \, t^{-\mu} \nonumber \\
						& =: c_1^c(u_0,f) \, t^{-\mu} \; , \label{C_(1)}
	\end{align}
	with
	\begin{align*}
		& \bullet \quad c_{-N}^c(u_0,f) := \frac{1}{2 \pi} \, \frac{1}{\mu} \left\| \tilde{u} \right\|_{L^{\infty}(-N,0)} + \frac{1}{2\pi} \, \Big( 4 \left\| \tilde{u} \right\|_{L^{\infty}(-N,0)} + \left\| \tilde{u}' \right\|_{L^1(-N,0)} \Big) \, m_{-N}^{-1} \; , \\
		& \bullet \quad c_{+N}^c(u_0,f) := \frac{1}{2 \pi} \, \frac{1}{\mu} \left\| \tilde{u} \right\|_{L^{\infty}(0,N)} + \frac{1}{2\pi} \, \Big( 4 \left\| \tilde{u} \right\|_{L^{\infty}(0,N)} + \left\| \tilde{u}' \right\|_{L^1(0,N)} \Big) \, m_{+N}^{-1} \; .
	\end{align*}
	Now we study the terms of the series. By hypothesis, $U$ is regular on $[n,n+1]$ for $n \in \mathfrak{S}_N$, $\psi'$ is monotone and is non-vanishing. Theorem \ref{VDC4} is then applicable once again on the interval $[n,n+1]$ with $\mu = 1$ and it furnishes
	\begin{equation*}
		\forall \, (t,x) \in \mathfrak{C}(a,b)^c \qquad \left| \int_{n}^{n+1} U(p) \, e^{it \psi(p)} \, dp \right| \leqslant c_n^c(u_0,f) \, t^{-1} \; ;
	\end{equation*}
	the constant $c_n^c(u_0,f) > 0$ is defined by
	\begin{equation*}
		c_n^c(u_0,f) := \frac{1}{2 \pi} \Big( 3 \left\| U \right\|_{L^{\infty}(n,n+1)} + \left\| U' \right\|_{L^1(n,n+1)} \Big) \, m_n^{-1} \; ,
	\end{equation*}
	with $\displaystyle m_{n} := \min \big\{ f'(n) - a , b - f'(n+1) \big\} > 0$. As in the previous proof, we can show by using the hypothesis on $u_0$ and Lemma \ref{LEM2} that
	\begin{equation*}
		\left\| U \right\|_{L^{\infty}(n,n+1)} \leqslant 2^{1-\mu + \alpha} M \, |n|^{\mu-1-\alpha} \; ,
	\end{equation*}
	and
	\begin{equation*}
		\left\| U' \right\|_{L^1(n,n+1)} \leqslant 2^{1-\mu} \left( 2^{\alpha} M + M' \right) |n|^{\mu-1-\alpha} \nonumber \; .
	\end{equation*}
	Furthermore the point ii) of Remarks \ref{REM6} implies
	\begin{equation*}
		m_n \geqslant \frac{c_-}{\beta_- - 1} \, \min \left\{ |n|^{1-\beta_-} , |n+1|^{1-\beta_-} \right\} \geqslant \frac{c_-}{\beta_- - 1} \, 2^{1-\beta_-} |n|^{1-\beta_-} \; ,
	\end{equation*}
	where we used Lemma \ref{LEM2} one more time. Then we obtain
	\begin{align*}
		c_n^c(u_0,f)	& = \frac{1}{2 \pi} \Big( 3 \left\| U \right\|_{L^{\infty}(n,n+1)} + \left\| U' \right\|_{L^1(n,n+1)} \Big) \, m_n^{-1} \\
					& \leqslant \frac{\beta_- - 1}{2 \pi c_-} \Big( 3 \times 2^{-\mu+\alpha+\beta_-} \, M \, + 2^{-\mu + \beta_-} \left(2^{\alpha} M + M' \right) \Big) \, |n|^{\mu-2-\alpha+\beta_-} \; .
	\end{align*}
	The summability of the sequence $\big\{ c_n^c(u_0,f) \big\}_{n \in \mathfrak{S}_N}$ comes from the assumption $\alpha > \mu + \beta_- - 1$, and we have
	\begin{align}
		\sum_{n \in \mathfrak{S}_N} c_n^c(u_0,f) 	& \leqslant \frac{\beta_- - 1}{ \pi c_-} \Big( 3 \times 2^{-\mu+\alpha+\beta_-} \, M \, + 2^{-\mu + \beta_-} (2^{\alpha} M + M') \Big) \, \frac{\alpha + 2 - \mu - \beta_-}{\alpha + 1 - \mu - \beta_-} \nonumber \\
													& =: c_2^c(u_0,f) \; . \label{C_2}
	\end{align}
	It follows
	\begin{equation*}
		\left| \sum_{n \in \mathfrak{S}_N} \int_n^{n+1} U(p) \, e^{it \psi(p)} \, dp \right| \leqslant \left( \sum_{n \in \mathfrak{S}_N} c_n^c(u_0,f) \right) t^{-1} \leqslant c_2^c(u_0,f) \, t^{-1} \; .
	\end{equation*}
	We obtain finally for all $(t,x) \in \mathfrak{C}(a,b)^c$,
	\begin{equation*}
		\big| u(t,x) \big| \leqslant c_1^c(u_0,f) \, t^{-\mu} + c_2^c(u_0,f) \, t^{-1} \; .
	\end{equation*}
\end{proof}

\section{Technical lemmas}

\hspace{0.5cm} In the last section, we state and prove two basic lemmas which are used several times in this paper.

\begin{LEM} \label{LEM}
	Let $\alpha \in (0,1]$ and let $x,y \in \R_+$ such that $x \geqslant y$. Then we have
	\begin{equation*}
		x^{\alpha} - y^{\alpha} \leqslant (x-y)^{\alpha} \; .
	\end{equation*}
\end{LEM}

\begin{proof}
	The case $\alpha = 1$ is trivial so let us assume $\alpha < 1$. If $y=0$ then the result is clear. Suppose $y \neq 0$, then the above inequality is equivalent to
	\begin{equation*}
		\left( \frac{x}{y} \right)^{\alpha} - 1 \leqslant \left( \frac{x}{y} - 1 \right)^{\alpha} \; .
	\end{equation*}
	Define the function $h : [1, +\infty) \longrightarrow \R$ by
	\begin{equation*}
		\forall \, z \in [1,+\infty) \qquad h(z) := (z-1)^{\alpha} - z^{\alpha} + 1 \; .
	\end{equation*}
	Then we note that for all $z > 1$,
	\begin{equation*}
		h'(z) = \alpha \left( (z-1)^{\alpha -1} - z^{\alpha-1} \right) \geqslant 0 \; ,
	\end{equation*}
	since $\alpha - 1 < 0$. It follows
	\begin{equation*}
		\forall \, z \in [1, +\infty) \qquad h(z) \geqslant h(1) = 0 \; ,
	\end{equation*}
	which proves the lemma.
\end{proof}

\begin{LEM2} \label{LEM2}
	Let $p \in [n,n+1]$, where $n \geqslant 1$ or $n \leqslant -2$. Then we have
	\begin{equation*}
		\frac{1}{2} \, |n| \leqslant |p| \leqslant 2 |n| \; .
	\end{equation*}
\end{LEM2}

\begin{proof}
	Firstly let us suppose that $n \geqslant 1$. Then
	\begin{equation*}
		\frac{1}{2} \, n \leqslant n \leqslant p \leqslant n + 1 \leqslant 2 \, n \; .
	\end{equation*}
	Now we suppose $n \leqslant -2$. Similar computations provide
	\begin{equation*}
		\frac{1}{2} \, |n| = - \frac{1}{2} \, n \leqslant -(n+1) \leqslant -p = |p| \leqslant -n = |n| \leqslant 2 \, |n| \; .
	\end{equation*}
\end{proof}

\end{document}